\documentclass[a4paper,11pt]{article}
\usepackage{amssymb}
\usepackage{graphicx}
\usepackage{amsthm}
\usepackage{bbm}
\usepackage{mathrsfs}
\usepackage{amsmath}
\usepackage{amsfonts}
\usepackage[hmargin=3.5cm, vmargin=4cm]{geometry}

\usepackage{color}
\newenvironment{DIFnomarkup}{}{} 

\def\expecti{\E}
\def\nexp{K_\expecti}
\def\bestnexp{\hat K_\expecti}
\def\bestzexp{\hat z_\expecti}

\def\asi{{\text{\upshape a.s.}}}
\def\nas{K_\asi}
\def\bestnas{\hat K_\asi}
\def\bestzas{\hat z_\asi}

\newtheoremstyle{mytheorem}{}{}{\itshape}{}{\bfseries}{:\hbox{ }}{\newline}{}
\newtheoremstyle{mydefinition}{}{}{}{}{\bfseries}{:\hbox{ }}{\newline}{}
\newtheoremstyle{myproof}{}{}{}{}{\bfseries}{:\hbox{ }}{\newline}{#1#3}

\theoremstyle{mytheorem}
\newtheorem{thm}{Theorem}[]

\newtheorem{cor}[thm]{Corollary}
\newtheorem{lem}[thm]{Lemma}
\newtheorem{prop}[thm]{Proposition}

\theoremstyle{mydefinition}

\theoremstyle{myproof}

\newtheorem{rmk}{Remark}

\renewcommand{\P}{\mathbb{P}}
\newcommand{\Pb}{\mathbb{P}}
\newcommand{\N}{\mathbb{N}}
\newcommand{\R}{\mathbb{R}}
\newcommand{\Rb}{\mathbb{R}}
\newcommand{\E}{\mathbb{E}}
\newcommand{\Eb}{\mathbb{E}}
\newcommand{\Pt}{\tilde{\mathbb{P}}}
\newcommand{\Qt}{\tilde{\mathbb{Q}}}

\newcommand{\Fg}{\mathcal{F}}
\newcommand{\Gg}{\mathcal{G}}
\newcommand{\Ft}{\tilde{\mathcal{F}}}
\newcommand{\Gt}{\tilde{\mathcal{G}}}
\newcommand{\hs}{\hspace{2mm}}
\newcommand{\hsl}{\hspace{1mm}}
\newcommand{\ind}{\mathbbm{1}}
\newcommand{\diffd}{\mathrm{d}}
\newcommand{\eps}{\varepsilon}

\author{Julien Berestycki\thanks{LPMA, UPMC, 4 place Jussieu, 75005 Paris, France. Email: \texttt{julien.berestycki@upmc.fr}}, \'Eric Brunet\thanks{LPS-ENS, UPMC, CNRS, 24 rue Lhomond, 75231 Paris Cedex 05, France. Email: \texttt{Eric.Brunet@lps.ens.fr}}, John W.~Harris\thanks{Department of Mathematics, University of Bristol, University Walk, Bristol BS8 1TW, UK. Email: \texttt{john.harris@bristol.ac.uk}},\\ Simon C.~Harris\thanks{Department of Mathematical Sciences, University of Bath, Bath BA2 7AY, UK. Email: \texttt{S.C.Harris@bath.ac.uk}}, Matthew I.~Roberts\thanks{Department of Mathematical Sciences, University of Bath, Bath BA2 7AY, UK. Email: \texttt{mattiroberts@gmail.com}}}

\title{Growth rates of the population in a branching Brownian motion with an inhomogeneous breeding potential}

\begin{document}

\maketitle

\begin{abstract}
We consider a branching particle system where each particle moves as an independent Brownian motion and breeds at a rate proportional to its
distance from the origin raised to the power $p$, for $p\in[0,2)$. The asymptotic behaviour of the right-most particle for this system is already known; in this article we give large deviations probabilities for particles following ``difficult'' paths, growth rates along ``easy'' paths, the total population growth rate, and we derive the optimal paths which particles must follow to achieve this growth rate.
\end{abstract}

\section{Introduction and heuristics}

\subsection{The model}

We study a branching Brownian motion (BBM) in an inhomogeneous breeding
potential on $\R$. Fix $\beta>0$, $p\in[0,2)$, and a random variable $A$,
which takes values in $\{1,2,\ldots\}$, satisfying $\E[A\log A]<\infty$. We initialise our branching process with a single particle at the origin. Each particle $u$, once born, moves as a Brownian motion, independently of all other particles in the population. Each particle $u$ alive at time $T$ dies with instantaneous rate $\beta |X_u(T)|^p$, where $X_u(T)$ is the spatial position of particle $u$ (or of its ancestor) at time $T$. Upon death, a particle $u$ is replaced by a random number $1+A_u$ of offspring in the same spatial position, where each $A_u$ is an independent copy of $A$.  We define $m:=\E[A]$, the average increase in the population size at each branching event. We denote by $N(T)$ the set of particles alive at time $T$.
We let $\P$ represent the probability law, and $\E$ the corresponding expectation, of this BBM.

The case $p=2$ is critical for this BBM: if the breeding rate were instead $\beta|\cdot|^p$ for $p>2$, it is known from It\^{o} and McKean~\cite{ito_mckean:diffusion_procs_sample_paths} that the population explodes in finite time, almost surely. For $p=2$, the \emph{expected} number of particles explodes in finite time, but the population remains finite, almost surely, for all time.

Branching Brownian motions are closely associated with certain partial differential equations. In particular, for the above BBM model, the McKean representation tells us that
\[
v(T,x):=\E \left( \prod_{u\in N(T)} f(x+X_u(T))\right)
\]
solves the equation
\begin{equation}
\frac{\partial v}{\partial T} = \frac12 \frac{\partial^2 v}{\partial x^2} +\beta |x|^p (G(v)-v)
\label{eqMcKean}
\end{equation}
with the initial condition $v(0,x)=f(x)$, where $G(s):=\E(s^A)$ is the generating function of the offspring distribution $A$. 
In the case of constant branching rate $(p=0)$, this is known as the  Fisher-Kolmogorov-Piscounov-Petrovski (FKPP) reaction-diffusion equation.

An object of fundamental importance in the study of branching diffusions is the right-most particle, defined as  $R_T:= \max_{u\in N(T)} X_u(T)$. 
Standard BBM, with binary branching at a constant rate (that is, $p=0$ and $G(s)=s^2$), has been much studied. 
In this case, it is well known that the linear asymptotic $\lim_{T\to\infty} R_T / T=\sqrt{2\beta}$ holds almost surely.
The distribution function of the right most particle position solves the FKPP equation with Heaviside initial conditions, 
and  it is known that $\P(R_T\geq m(T) + x) \rightarrow w(x)$ where $w$ is a travelling-wave solution of \eqref{eqMcKean} and $m(T)$ is the median for the rightmost particle position at time $T$.
Sub-linear terms for the asymptotic behaviour of $m(T)=\sqrt{2\beta} T - 3/(2\sqrt{2\beta}) \log T + O(1)$ were found by Bramson~\cite{bramson:maximal_displacement_BBM} and~\cite{bramson:convergence_Kol_eqn_trav_waves}. 
See also the recent shorter probabilistic proofs by Roberts~\cite{roberts:simple_path_BBM}, and corresponding results for branching random walk by Aidekon \cite{aidekon:convergence_law_min_brw} and Hu and Shi \cite{hu_shi:minimal_position}.
For approaches using partial differential equation theory, see the recent short proof by Hamel {\em et al.~}\cite{hamel_et_al} and an impressive higher order expansion due to Van Saarloos \cite{van_saarloos}.
Detailed studies of the paths followed by the  right-most particles have been carried out by Arguin {\em et al.~}\cite{arguin_et_al:extremal_BBM, arguin_et_al:genealogy_BBM}, and by Aidekon {\em et al.~}\cite{aidekon_et_al:BBM_tip}. 

For $p\in(0,2)$, right most particle speeds much faster than linear occur and Harris and Harris \cite{harris_harris:inhom_breeding} found an asymptotic for $R_T$ using probabilistic techniques involving additive martingales and changes of measure.

\begin{thm}[Harris, Harris \cite{harris_harris:inhom_breeding}]\label{rightmost_thm}
For $p\in [0,2)$,
\[\lim_{T \to\infty} \frac{R_T}{T^\frac{2}{2-p}} = \left(\frac{m\beta}{2}(2-p)^2\right)^{\frac{1}{2-p}}\]
almost surely.
\end{thm}
\noindent (Note that, as above, the theorems given in the present paper are all written for $p\in[0,2)$. However most of our results were already known in the case $p=0$.)

In this paper we study in more detail the paths followed by particles in the BBM. Theorem~\ref{rightmost_thm} suggests a rescaling of time and space, and we consider whether particles follow paths which, after rescaling, lie in a particular subset of $C[0,1]$. In section~\ref{S:main results}, we give large deviations probabilities for particles following ``difficult" rescaled paths 
as well as results on the almost sure growth rates for the number of particles following any ``easy" path. From these results we can derive the growth rate of the total number of particles in the BBM, and find the paths which particles must follow to realise this growth rate; this involves solving certain path optimisation problems subject to integral constraints,  the solutions to which are not obvious, but nevertheless can be found explicitly and have intuitive probabilistic interpretations. A surprising and very significant feature arising from the (unbounded) spatially dependent branching rate of this model is the fact that the expected number of particles and typical number of particles following paths do not match, even on the exponential scale.

Although this work is the natural sequel to \cite{harris_harris:inhom_breeding}, spatially dependent branching rates have not often been studied in detail. 
See Git {\em et al.~}\cite{git_et_al:exp_growth_rates}, and Lalley and Sellke \cite{LS1, LS2} for a case with bounded breeding potential. Other studies of branching processes with time inhomogeneous environments include recent works by Fang and Zeitouni \cite{fang_zeitouni:brw_inhom, fang_zeitouni:slowdown}, where analogous path optimisation problems also appear.
Recent developments in the study of spatially inhomogenous versions of the FKPP equation from the PDE's perspective include, for example, the {periodic} environments in Hamel {\emph et al.~} \cite{hamel_et_al2}. A key technique which is used in \cite{hamel_et_al, hamel_et_al2} is to relate the non-linear PDE problem to a free boundary linearised PDE one. In fact, such free boundary problems are intimately related to the probabilistic constrained path optimisation problems (discussed in Section \ref{S:pdelinks}).
We note that, with the exception of \cite{harris_harris:inhom_breeding, git_et_al:exp_growth_rates}, the above articles are all concerned with \emph{bounded} environments.

Unbounded branching rates lead to unusual features and pose considerable technical difficulties, much as their corresponding unbounded non-linear differential operators would. 
One manifestation of the unbounded branching rates is the position of the right-most particle growing faster than linearly in time (as in Theorem \ref{rightmost_thm});
another is the disagreement of expected and typical particle behaviours. 

We start by giving a very rough heuristic explanation for some of our results. 
The technical details, relevant definitions and precise statements of our main results will be given in section \ref{S:main results}.
We have strived to make the heuristics as clear as possible in the hope that the reader can gain a good understanding of our main results without necessarily having
to read the technical details in the rigorous proofs of sections \ref{spine_section}-\ref{large_devs_sec}. 
The solution of the constrained path optimisation problem and it properties are found in sections \ref{sec:optpaths} and \ref{further_optpaths_sec}.


\subsection{Heuristics}

Whilst we are mainly interested in the almost sure behaviour of the inhomogeneous branching Brownian motion, this is typically much harder to obtain than the expected behaviour.
However, we will be able to get a very good intuitive understanding of the {almost sure} behaviour by carefully considering {expectations} of the number of particles that travel close to given trajectories.

In a sense that Schilder's Theorem \cite{varadhan:large_devs_apps} from Large Deviation theory can make precise,
the small probability that a Brownian motion $B$ manages to stay `close' to some given trajectory $F$ is {very roughly} given by
\[
\P(B \hbox{ stays `close' to } F \hbox{ during } [0,T]) \approx \exp\left( -\frac{1}{2}\int_0^T F'(s)^2 \diffd s\right).
\]
where $F:[0,T]\to\R$ is suitably `nice' with $F(0)=0$ and $T$ is some large time.

Since any particles that are close to trajectory $F$ at time $s$ will give birth to an average of $m$ new offspring at a rate close to $\beta |F(s)|^p$ ,
the expected total number of particles that have stayed close to trajectory $F$ up to time $T$ will very roughly be given by
\begin{equation}
\label{E:approx1}
\E\big[\#\{u\in N(T) : X_u \text{ `close' to } F \}\big] \approx \exp\left( \int_0^T \left[m\beta|F(s)|^p - \frac{1}{2} F'(s)^2\right] \diffd s \right).
\end{equation}
Heuristically, we can think of the number of particles travelling along
a `nice' trajectory $F$ as behaving roughly like a time-dependent
birth-death process (see~\cite{Harris:branching_processes})
with a birth rate $m\beta|F(s)|^p$ and a death rate $\frac{1}{2} F^\prime(s)^2$.
It is now natural to look for a scaling of paths where the birth and death rates are of the same order of magnitude.
That is, if we consider trajectories of the form
 \begin{equation}
 F(s) = T^{\frac{2}{2-p}} f\left(\frac s T\right) 
 \label{rescale}
 \end{equation}
where $T$ is large and $f:[0,1]\to\R$ is some fixed function, then \eqref{E:approx1} leads to
\begin{equation}\label{approx_rescaled}
\log \E\big[\#\{u\in N(T) : u \text{ follows } f\}\big] \sim T^{\frac{2+p}{2-p}}\int_0^1\left[m\beta |f(s)|^p - \frac{1}{2} f'(s)^2\right] \diffd s,
\end{equation}
where by `$u$ follows $f$' we mean that the rescaled particle path, $T^{-2/(2-p)}X_u(sT)$, remains within some very small distance of the given rescaled path $f(s)$ for all $s\in[0,1]$.
Essentially, this is Theorem~\ref{expected_growth_paths_thm} which reveals how the expected number of particles varies along a given scaled up path: heuristically,
for $t\in[0,1]$ and $T$ large,
\begin{equation}
\log \E\big[\#\{u\in N(tT) : \text{$u$ follows $f$ up to time $t$}\}\big]
\sim T^{\frac{2+p}{2-p}} K(f,t)
\label{heuristic5}
 \end{equation}
where
\begin{equation}
K(f,t):=\int_0^t \Big[m\beta |f(s)|^p - \frac{1}{2} f'(s)^2\Big] \diffd s  
\end{equation}
for `nice' functions $f:[0,1]\rightarrow\R$ with $f(0)=0$.
The functional $K(f,t)$ is of fundamental importance for the inhomogeneous BBM and will play a crucial role throughout this paper.

We will show that the expected number of particles which end up near the rescaled position $z$ (corresponding to actual position $T^{2/(2-p)}z$)
grows like the expected number of particles following some optimal rescaled path $h_z$ ending at $z$,
\begin{equation}
\log\E\big[\#\{u\in N(T) : \text{rescaled path of $u$ ends near $z$}\}\big]
\sim {T^{\frac{2+p}{2-p}}K(h_z,1)},
\end{equation}
where $K(h_z,1)=\sup_{f}\big\{K(f,1): f(1)=z\big\}$. 
The optimal function $h_z$ satisfies the corresponding Euler-Lagrange equation $h_z^{\prime\prime} + mp h_z^{p-1}=0$ with $h_z(0)=0$ and $h_z(1)=z$.
Optimising over $z$ then suggests that the expected total population size satisfies
\begin{equation}
\log\E\big[N(T)\big]
\sim {T^{\frac{2+p}{2-p}} \sup_{z} K(h_z,1)}
= {T^{\frac{2+p}{2-p}} \sup_{f} K(f,1)}.
\end{equation}
Theorem~\ref{unconstrained_thm} confirms these heuristics 
and explicitly identifies the expected total population growth rate.

However, a more surprising fact is that these results on the expected number of particles are {\emph not}  representative of a typical realisation of the system.
Indeed, a dominant contribution to the expected number of particles at large time $T$ can come from vanishingly rare events where particles go very far away from the origin to take advantage of high reproduction rates.
Note that, for any $t\in[0,1]$,
\begin{align}
\P(\text{some particle follows $f$ up to $t$}) &= \inf_{s\in[0,t]}\P(\text{some particle follows $f$ up to time $s$}) \notag\\
& \leq \inf_{s\in[0,t]} \E\big[ \#\{ u\in N(sT) : \text{$u$ follows $f$ up to $s$}\} \big],
\notag\\& \lesssim \exp\Big[T^{\frac{2+p}{2-p}}\inf_{s\in[0,t]}K(f,s)\Big].
\label{previouslyunumbered}
\end{align}
If $\inf_{s\in[0,t]}K(f,s)<0$, then the path $f$ is `difficult' and it is very unlikely to be observed in a typical realisation, even if $K(f,1)>0$ so that the \emph{expected} number of particles alive at time~$T$ having followed that path is very large.
We are in fact able to show  that, for a difficult path, the last inequality in~\eqref{previouslyunumbered} will be attained up to leading order in the exponent, and hence that `difficult' paths will satisfy
\begin{equation}
\log \P\left(\text{some particle follows $f$ up to $t$}\right)
\sim T^{\frac{2+p}{2-p}}\inf_{s\in[0,t]}K(f,s) < 0.
\end{equation}
This probability of presence result is stated rigorously in Theorem~\ref{large_devs_thm}.
Looking for the trajectory that travels the furthest without being `difficult' leads us to guess that the right-most particle boundary
satisfies $m\beta r(s)^p - \frac{1}{2} r'(s)^2\equiv 0$ with $r(0)=0$ (in agreement with Theorem~\ref{rightmost_thm}).

On the other hand, if we have not yet had any `difficult' points along the path, we might guess that the almost sure and the expected growth rates will still agree.
Indeed, Theorem~\ref{growth_paths_thm} will confirm that, roughly speaking, we almost surely have
\begin{multline}
\log \#\{u\in N(tT) : \text{$u$ follows $f$ up to $t$}\}\\\sim
T^{\frac{2+p}{2-p}} K(f,t)\qquad\text{if $K(f,s)\ge0$ for all $s\in[0,t]$}
.
\end{multline}
for any $t\in[0,1]$. Hence, if $K(f,s)$ is always non-negative there will almost surely be some particles following $f$ in the large $T$ limit.
On the other hand, if $K(f,s)$ becomes strictly negative for the first time at some time $\theta_0$ then it becomes exponentially unlikely that any particle makes it past this bottleneck;
$\theta_0$ corresponds to the extinction time along this `difficult' path $f$.

Finally, we anticipate that the almost sure number of particles which end up near the rescaled position $z$, with $|z|\le r(1)$, will grow like the expected number of particles following some optimal path $g_z$ that does not undergo any extinction:
\begin{equation}
\log\#\{u\in N(T) : \text{rescaled path of $u$ ends near $z$}\}
\sim {T^{\frac{2+p}{2-p}}K(g_z,1)},
\end{equation}
where $K(g_z,1)=\sup_f\big\{K(f,1) : f(1)=z, K(f,s)\ge0\text{ for all
\(s\in[0,1]\)}\big\}$.
The optimal path $g_z$ that gives rise to the vast majority of particles turns out to have two distinct phases.
Initially, it follows the trajectory of the right-most particle, thereby gaining optimal potential for future growth without becoming extinct.
Then, after some optimal intermediate time, it ``cashes in'' during the second phase, switching over to the path that maximises growth, which is the path
that satisfies $g_z^{\prime\prime} + mp g_z^{p-1}=0$ with $g_z(1)=z$. We will see that, in addition, the optimal path necessarily has a continuous derivative and this property determines the point at which one switches from one phase to the other.

The almost sure total number of particles in the system can now be recovered by optimising over $z$, with
\begin{equation*}
\log |N(T)| \sim T^{\frac{2+p}{2-p}} \sup_z K(g_z,1) = T^{\frac{2+p}{2-p}} \sup_f\big\{K(f,1) : K(f,s)\ge0\text{ for all $s\in[0,1]$}\big\}.
\end{equation*}
In particular, this will reveal that the almost sure population growth rate is strictly smaller than the expected population growth rate;
the constraint that the paths cannot have passed through any extinction times has a significant effect.


%
%
%
%

\section{Main results} \label{S:main results}

Fix a set $D\subset C[0,1]$ and $t\in[0,1]$. We are interested in the sets
\begin{equation}
N_T(D,t) := \left\{u \in N(t T) : \exists f \in D \hbox{ with } X_u(sT) = T^{\frac{2}{2-p}} f(s) \hs \forall s\in[0,t]\right\}
\label{NTD}
\end{equation}
for large $T$. We will typically consider sets of the form $D = B(f,\epsilon)$ for a given $f$ (the ball of $C[0,1]$ with centre $f$ and radius $\epsilon$)\footnote{In this paper, the space $C[0,1]$ of continuous functions on $[0,1]$ is always endowed with the $L_\infty$ topology.};
in this case $N_T(D,t)$ is the set of particles alive at time $t T$ whose rescaled paths up to that point have stayed within distance $\varepsilon $ of $f$.
Thus, for a given path $f$ which we keep rescaling in space and time according to large~$T$, $N_T(D,t)$ tells us how the population following that path grows and shrinks as $t$ varies between $0$ and $1$.

Define the class $H_1$ of functions by
\[H_1 := \left\{ f \in C[0,1] : \exists g \in L^2[0,1] \hbox{ with } f(t) = \int_0^t g(s)\, \diffd s \hs \forall t\in[0,1]\right\},\]
and to save on notation set $f'(s):=\infty$ if $f\in C[0,1]$ is not differentiable at the point $s$.
Observe that $f \in H_1$ implies $f(0)=0$.

We can now define $K$ precisely: for $f\in C[0,1]$ and $t\in[0,1]$,
\[
K(f,t) := \begin{cases}
\displaystyle\int_0^t \Big[ m\beta |f(s)|^p - \frac{1}{2} f'(s)^2 \Big]\diffd s   & \text{ if $f\in H_1$},
\\ - \infty & \text{otherwise}.\end{cases}
\]

We use throughout the paper the convention that $\inf\emptyset=+\infty$ and
$\sup\emptyset=-\infty$.

\subsection{Expected population growth}

Our first result is rather straightforward and gives the behaviour of the expectation of the number of particles following paths in some set.
\begin{thm}\label{expected_growth_paths_thm}
For any closed set $D\subset C[0,1]$ and $t\in[0,1]$,
\[\limsup_{T\to\infty} \frac{1}{T^{\frac{2+p}{2-p}}}\log\E|N_T(D,t)|\leq \sup_{f\in D} K(f,t),\]
and for any open set $A\subset C[0,1]$ and $t\in[0,1]$,
\[\liminf_{T\to\infty} \frac{1}{T^{\frac{2+p}{2-p}}}\log\E|N_T(A,t)|\geq \sup_{f\in A} K(f,t)\]
\end{thm}
Moreover, if we define
\begin{equation}
\nexp(z):=\sup \big\{ K(f,1) : f\in C[0,1], \hs f(1)=z\big\},
\label{nexp}
\end{equation}
we have the following easy corollary:
\begin{cor}\label{cor_unconstrained}
For each $\varepsilon >0$ and $z\in \R$, let $D_{z,\epsilon} := \{ f \in
C[0,1]: |f(1)-z| \le \epsilon\}.$ Then
$$
\lim_{\epsilon\rightarrow 0} \lim_{T \rightarrow\infty}  \frac1{T^{\frac{2+p}{2-p}}} \log \E |N_T(D_{z,\epsilon}, 1)| = \nexp(z).
$$
\end{cor}
Therefore, $\nexp(z)$ controls the growth rate of the expectation of the number of particles which end up near $z$ on the rescaled space. The next theorem shows that the supremum defining $\nexp(z)$ corresponds to a unique optimal path $h_z$; optimising over $z$ then gives the total expected population growth.

\begin{thm}\label{unconstrained_thm}
For $z\in\R$, the optimisation problem
\[\nexp(z)= K(h_z,1)\]
has a solution $h_z\in C^2[0,1]$ which is unique for $z\ne0$ amongst all $H_1$ functions ending at $z$. For $z\ge0$, the solution $h_z$ is positive and satisfies for all $s\in[0,1]$
\[h_z''(s) + m\beta p h_z(s) ^{p-1} = 0, \hs\hs h_z(0)=0, \hs\hs h_z(1) = z.\]
Furthermore there exists a unique $\bestzexp \geq0$ such that
\[
\bestnexp:=\nexp(\bestzexp)=\sup_z \nexp(z)=\sup_{f\in C[0,1]} K(f,1).
\]
Then the expected total population size satisfies
\[\lim_{T\to\infty}\frac{1}{T^{\frac{2+p}{2-p}}}\log \E|N(T)| =
\bestnexp,
\]
where one finds
\[ h_{\bestzexp}'(1)=0,
\qquad\bestzexp
=\frac{(2m\beta)^{\frac1{2-p}}}{\left[\int_0^1\frac{\diffd x}{\sqrt{1-x^p}}\right]^{\frac2{2-p}}}
= (2m\beta)^{\frac1{2-p}}\left[\frac{\Gamma\big(\frac12+\frac1p\big)}{\sqrt{\pi}\,\Gamma\big(1+ \frac1p\big)}\right]^{\frac2{2-p}}
\]
and
\[\bestnexp=\frac{2-p}{2+p}m\beta\bestzexp^p.\]
\end{thm}
\begin{rmk}
For $z<0$, one has $h_z(s)=-h_{-z}(s)$. For $z=0$ and $p>0$, there are two symmetrical optimal paths, one positive and one negative. For $z=0$ and $p=0$ the optimal path is unique and equal to $h_0=0$.
\end{rmk}

\subsection{Almost sure growth along paths}\label{as_growth_section}


Let us now focus on the problem of giving an almost sure result for the actual number of particles that have a rescaled path lying in some set $D$.

We let
\[\theta_0(f):= \inf\left\{t \in[0,1] : K(f,t) < 0  \right\}\in [0,1)\cup\{\infty\} .\]
We think of $\theta_0$ as the extinction time along $f$, the time at which the number of particles following $f$ hits zero:
if $t > \theta_0(f)$, basically at large times no particle has a path that looks like $f$ up to time $t$. On the other hand, if $t \le \theta_0(f)$, the number of particles with a rescaled path looking like $f$ up to time~$t$ grows like the expected number of particles following that path.
This is made precise in Theorem \ref{growth_paths_thm} below.

\begin{thm}\label{growth_paths_thm}
For any closed set $D\subset C[0,1]$ and $t \in[0,1]$,
\[\limsup_{T\to\infty} \frac{1}{T^{\frac{2+p}{2-p}}}\log|N_T(D,t)|   \leq \sup \{  K(f,t) : f\in D, \theta_0(f) \ge t\}  \hs\hbox{almost surely,}\]
and for any open set $A\subset C[0,1]$ and $t\in[0,1]$,
\[\liminf_{T\to\infty} \frac{1}{T^{\frac{2+p}{2-p}}}\log|N_T(A,t)|\geq \sup
\{  K(f,t) : f\in A, \theta_0(f) \ge t\}\hs\hbox{almost surely}.\]
\end{thm}



Moreover, if one defines
\begin{equation}\label{Kasconstraint}
\nas(z):=\sup \big\{ K(f,1), f\in C[0,1],f(1)=z, \theta_0(f)=\infty\big\},
\end{equation}
we obtain the following corollary:
\begin{cor}\label{C:corol nas}
For $z\in  \R$,
$$
\lim_{\epsilon\rightarrow 0} \lim_{T \rightarrow\infty}  \frac1{T^{\frac{2+p}{2-p}}} \log |N_T(D_{z,\epsilon}, 1)| = \nas(z)\qquad\text{almost surely}.
$$
\end{cor}
Therefore, $\nas(z)$ controls the growth rate of the almost sure number of particles which end up near $z$ on the rescaled space. The next theorem shows that the supremum defining $\nas(z)$ corresponds to a unique optimal path $g_z$ (that, therefore, most particles ending up near $z$ must have followed); optimising over $z$ then yields the almost sure total population size growth.
Let
\begin{equation}
r(s) := \left(\frac{m\beta s^2}{2}(2-p)^2\right)^{\frac{1}{2-p}}\quad\text{and}\quad \bar z :=r(1)=\left(\frac{m\beta}{2}(2-p)^2\right)^{\frac{1}{2-p}}.
\label{r(s)}
\end{equation}
Observe that by Theorem~\ref{rightmost_thm}, for all $s\in[0,1]$,
$R_{sT}/T^{2/(2-p)} \to r(s)$ almost surely as $T\to\infty$. This means that $r(s)$ describes the boundary of the limiting shape of the trace of the rescaled BBM and $\bar z$ is the rescaled position of the right-most particle at time 1.

\begin{thm}\label{constrained_thm}
For each $z \in [-\bar z, \bar z]$, one has
\begin{equation}\label{Kaskilling}
\nas(z)=\sup\{K(f,1) : f\in C[0,1], f(1)=z, |f(s)|\le r(s)\ \forall s\in [0,1]\}.
\end{equation}
Moreover, the optimisation problem
\[\nas(z)=K(g_z,1)\]
has a solution $g_z$ which is unique for $z\ne0$ amongst all $H_1$ functions ending at $z$ such that $\theta_0(\cdot)=\infty$.
For $|z|>\bar z$ one has $\nas(z)=-\infty$, which means that no function of $H_1$ with $\theta_0(\cdot)=\infty$ reaches $z$.

The solution $g_z$ for $0\le z\le\bar z$ is characterised as follows: there exists a unique $s_z \in[0,1]$ such that
\begin{enumerate}
\item for all $s\in[0,s_z]$, $g_z(s)=r(s)$
\item for all $s\in(s_z,1]$, $g_z$ is twice continuously differentiable and
\begin{equation}\label{E:diff eq for g_z}
g_z''(s) + m\beta p g_z(s)^{p-1} = 0, \hs\hs g_z(1)=z;
\end{equation}
\item $g_z$ is differentiable at $s_z$.
\end{enumerate}
Furthermore there exists a unique $\bestzas\ge0$  such that
\[\bestnas:=\nas(\bestzas)=\sup_z \nas(z)=\sup \big\{ K(f,1), f\in C[0,1], \theta_0(f)=\infty\big\}.\]
Then the almost sure total population size satisfies
\[\lim_{T\to\infty}\frac{1}{T^{\frac{2+p}{2-p}}}\log |N(T)| = \bestnas\qquad\text{almost surely},
\]
where one finds
\[ g_{\bestzas}'(1)=0,
\qquad \bestzas = \left[\frac{\sqrt{2m\beta}}{\frac {2^{\frac{3p-2}{2p}}}{2-p}
	+\displaystyle\int_{2^{-1/p}}^1\frac{\diffd x}{\sqrt{1-x^p}}}\right]^\frac2{2-p}
\]
and
\[\bestnas=\frac{2-p}{2+p}m\beta\bestzas^p.\]

\end{thm}
\begin{rmk}
It is easy to see that for $z\ge 0$ we need only consider positive functions since for a general $g, K(|g|,t) = K(g,t)$ for all $t\in[0,1]$, and it is not hard to see that  for all $t,z>0$, $g_z(t)>0$. For $p=0$, one has $s_z=0$ $\forall z \in[0,\bar z)$ and the almost-sure and expectation paths coincide.
When $p>0$, the proofs will make clear that $s_z>0$ $\forall z \in [0,\bar z]$. In particular this means  (still when $p>0$) that the majority of particles found near the origin have in fact followed either the left-most or right-most path for some proportion of their history and travelled a long way out before increasing in number whilst heading back away from the frontier.
\end{rmk}


An easy consequence of the theorem is that $r(s)$ describes not only the
limiting trace of the BBM, but also the actual rescaled trajectory of the
rightmost particle at time~$T$. It is the trajectory on the onset of
extinction, the one for which $K(r,t)=0$ for all $t$ or, equivalently, for
which
\begin{equation}
\frac{1}{2}r'(s)^2 = m\beta r(s)^p,\qquad\text{with $r(0)=0$},
\label{eqr(s)}
\end{equation}
as can be directly checked from~\eqref{r(s)}. It is interesting to observe that $r$ thus satisfies $r''(s) = m\beta p r(s)^{p-1}$. The solution $g_z$ thus satisfies  the same second-order differential equation on $[0,s_z)$ and $(s_z,1]$ up to a sign difference on the second term.

\subsection{Probability of presence}
\label{large_devs_section}

If $f$ is such that $\theta_0(f)<1$, Theorem~\ref{growth_paths_thm} suggests that as $T$ becomes large the number of particles whose rescaled paths have stayed close to $f$ is almost certainly 0. The following large deviations result shows how the probability of presence of a particle close to $f$ up to time $t$ decreases as $t$ goes from $0$ to $1$.
\begin{thm}\label{large_devs_thm}
For any closed set $D\subset C[0,1]$ and $t\in [0,1]$,
\[\limsup_{T\to\infty} \frac{1}{T^{\frac{2+p}{2-p}}}\log\P( N_T(D,t) \neq \emptyset ) \leq
\sup_{f\in D}\left[\inf_{s\le t} K(f,s)\right]\]
and for any open set $A\subset C[0,1]$ and $t \in [0,1]$,
\[\liminf_{T\to\infty} \frac{1}{T^{\frac{2+p}{2-p}}}\log\P( N_T(A,t) \neq \emptyset ) \geq
\sup_{f\in A}\left[ \inf_{s\le t} K(f,s)\right]\]
\end{thm}

The case $p=0$ was proved by Lee \cite{lee:large_deviations_for_branching_diffusions} and again by Hardy and Harris \cite{hardy_harris:spine_bbm_large_deviations}.

\subsection{Relationship to differential equations}\label{S:pdelinks}

In this section we try to show how several of our results can actually be guessed from heuristic manipulations of partial differential equations. Although the whole discussion is informal it leads us to a theorem which gives an alternative description of $\nas$ and $\nexp$.

The expected density $\rho(x,T)$ of points at position $x$ and time $T$ in
the BBM we are studying can be written as the solution of the partial differential equation
\begin{equation}
\frac{\partial\rho}{\partial T} = \frac12 \frac{\partial^2\rho}{\partial x^2} +m\beta |x|^p \rho.
\label{eqrho}
\end{equation}
Corollary~\ref{cor_unconstrained} suggests that for large~$T$,
\begin{equation}
\log\rho(x,T) \sim
T^{\frac{2+p}{2-p}}
\nexp(z)
\quad\text{with $z=\dfrac x {T^{\frac{2}{2-p}}}$}.
\label{ansatz}
\end{equation}
If we then
plug~\eqref{ansatz} into~\eqref{eqrho} we get a differential equation: for $T$ large, neglecting a term of order $T^{-(2+p)/(2-p)}\nexp''(z)$, we get
\begin{equation}
\frac12\left(\nexp'(z)+\frac{2z}{2-p}\right)^2=\frac{2+p}{2-p}\nexp(z)+\frac{2z^2}{(2-p)^2}-m\beta|z|^p.
\label{eqK}
\end{equation}

It is not obvious at first that the $\nexp(z)$ defined in \eqref{nexp} is
indeed a solution of~\eqref{eqK} but we will show that this is the case in Theorem~\ref{constrained_id} below.
The differential equation, however, is not enough to fully obtain $\nexp(z)$ as there is no obvious initial condition.

A natural question is now: is there a differential equation of which $\nas(z)$
is solution? $\nas(z)$ describes the growth rate of the almost sure number of
particles at rescaled position $z$ and it is different from the expected growth rate $\nexp(z)$
because of some extremely rare events (on which particles go far away and reproduce a lot) which contribute to $\nexp(z)$
and not to $\nas(z)$ in the $T\to\infty$ limit.

With this in mind, we now consider the inhomogeneous BBM with killing, where we remove any particles
that ever cross the two space-time boundaries $(s,\pm \bar x(s))_{s\geq 0}$, for some given function $\bar x(\cdot)$.
The expected density $\tilde \rho(x,T)$ of particles at position $x$ at time
$T$ in this BBM with killing is then a solution of
\begin{equation}
\begin{cases}\displaystyle\frac{\partial\tilde\rho}{\partial T}
=\frac12\frac{\partial^2\tilde\rho}{\partial x^2}+m\beta |x|^p \tilde\rho,
\\[1ex]
\displaystyle \tilde\rho\big(\pm \bar x(T),T\big)=0.
\end{cases}
\label{absorb}
\end{equation}
If the absorbing boundary $\bar x(\cdot)$ is taken to be the typical trajectory of the
right-most particle, $r(\cdot)$, (see Theorem \ref{rightmost_thm} and equation \eqref{r(s)}), this in effect kills all
those rare difficult paths and one might hope heuristically that
$\tilde\rho(x,T)$ describes, in some sense, the almost sure density of points in the
original problem without the absorbing boundaries:
\begin{equation}
\log\tilde\rho(x,T) \sim
T^{\frac{2+p}{2-p}} \nas(z)
\quad\text{with $z=\dfrac x {T^{\frac{2}{2-p}}}$}.
\label{ansatz2}
\end{equation}
In fact, this heuristic does turn out to be the case.
Indeed, the a.s. growth rate, $\nas(z)$, is given by maximising over paths that end at $z$ and do not undergo extinction at any point, as in \eqref{Kasconstraint} and Corollary~\ref{C:corol nas}.
However,  $\nas(z)$ can also be expressed by maximising over paths that end at $z$ and never go beyond the right/left-most paths, as in \eqref{Kaskilling} of Theorem~\ref{constrained_thm}.
This equivalent representation is exactly the same as would be obtained for the expected growth rate in the BBM with killing at $\pm r(s)$,
hence $\nas(z)$ is indeed given by equation \eqref{ansatz2} where $\tilde\rho$ satisfies PDE \eqref{absorb}.

Further, we might even hope to determine $r(\cdot)$ in a self-consistent way: if, in
\eqref{absorb}, $\bar x(T)$ is significantly smaller than the almost sure
position of the right-most particle, then many particles will gather close to
the line and we can expect $\tilde\rho(\bar x(T)-1,T)$ to be large. On the
other hand, if $\bar x(T)$ is chosen significantly larger than the almost sure
position of the right-most particle, then very few particles should come
close to the boundary and $\tilde\rho(\bar x(T)-1,T)$ should be small. Only
for $\bar x(T)$ close to the almost sure
position of the right-most particle can we expect $\tilde\rho(\bar x(T)-1,T)$
to be of order one. We therefore reformulate \eqref{absorb} into
\begin{equation}
\begin{cases}\displaystyle\frac{\partial\tilde\rho}{\partial T}
=\frac12\frac{\partial^2\tilde\rho }{\partial x^2}+m\beta |x|^p \tilde\rho,
\\[1ex]
\displaystyle \tilde\rho\big(\pm \bar x(T),T\big)=0,\\[1ex]
\displaystyle \frac{\partial\tilde\rho}{\partial x}(\pm \bar x(T),T\big)=\pm1,
\end{cases}
\label{absorb2}
\end{equation}
where we solve now simultaneously for the two unknowns $\tilde\rho$ and
$\bar x$.

We now plug \eqref{ansatz2} into \eqref{absorb2} and, as the equations in
the bulk for $\rho$ and $\tilde\rho$ are the same, we obtain the
same equation~\eqref{eqK} for $\nas(z)$ as for $\nexp(z)$, albeit with different boundary conditions.
From (\ref{absorb2}) we see that $\tilde \rho(x,T)=0$ for $x >\bar x(T)$. Thus (\ref{ansatz2}) means that  $\nas(z)$ is not defined for $z$ above some threshold value.
Therefore we look for a solution of (\ref{eqK}) only defined up to a finite value of $z$, which can only be the  rescaled almost sure position $\bar z$ of the right-most particle, as defined in (\ref{r(s)}).
The only way for this to happen is for the right-hand side of \eqref{eqK} to vanish at
$z=\bar z$:
\begin{equation}\label{rhs}
\frac{2+p}{2-p}\nas(\bar z)+\frac{2\bar z^2}{(2-p)^2}-m\beta|\bar z|^p=0.
\end{equation}
(When the right-hand side reaches 0, one can check that the second derivative diverges and the solution cannot be continued beyond that point.)
Furthermore, one must have $\nas(\bar z)=0$ from the second boundary condition in (\ref{absorb2}); indeed $\nas(\bar z)>0$ would correspond to having increasingly many particles at the boundary while $\nas(\bar z)<0$ would mean that the number of particles next to the boundary goes to zero. We thus recover the expression for $\bar z$ given in~\eqref{r(s)} as a solution of \eqref{rhs}.

From the descriptions of $\nexp(z)$ and $\nas(z)$ given in Theorems~\ref{unconstrained_thm} and~\ref{constrained_thm}, one can show that, indeed,
\begin{thm}\label{constrained_id}
$\nexp(z)$ and $\nas(z)$ as defined in Theorems~\ref{unconstrained_thm}
and~\ref{constrained_thm} are, for $z\ge0$,  two solutions of the following
differential equation:
\begin{equation}
K'(z)=-\frac{2z}{2-p} +
\sqrt{2\frac{2+p}{2-p}K(z)+\frac{4z^2}{(2-p)^2}-2m\beta z^p}.
\label{eqK2}
\end{equation}
\end{thm}
(Note that \eqref{eqK2} is not implied by \eqref{eqK} as we could have put
a minus sign in front of the square root. We will show in the proof section
that $K'(0^+)\ge0$ which justifies the choice of the plus sign. For $z\le0$,
by parity of $K(z)$, the other sign must be chosen.)

This approach from partial differential equations can be extended to the
case $p=2$. Considering only the particles that do not go further away than
the almost sure position of the right-most, we start from \eqref{absorb2}
with $p=2$. The scaling function \eqref{ansatz2} obviously does not work
for $p=2$, but if one plugs
\begin{equation}
\log\tilde\rho(x,T)\approx e^{2A T} L(z)\qquad\text{with $z=x e^{-A T}$}
\label{ansatz3}
\end{equation}
into~\eqref{absorb2} one gets that
\begin{equation}
\frac12\big(L'(z)+Az\big)^2= 2A L(z)+\frac12A^2z^2-m\beta z^2
\label{eqL}
\end{equation}
where a term of order $e^{-2AT}L''(z)$ has been neglected. \eqref{eqL} has
exactly the same structure as \eqref{eqK} except that $A$ is {\em a priori} an
unknown quantity. However, requesting as for $p<2$ that the right-hand
side vanishes at $z=\bar z$ and that $L(\bar z)=0$ implies that
\begin{equation}
A=\sqrt{2m\beta}
\label{predicA}
\end{equation}
so that the equation reads
\begin{equation}
\frac12\big(L'(z)+\sqrt{2m\beta}\,z\big)^2=2\sqrt{2m\beta}\,L(z),
\label{eqL2}
\end{equation}
or, taking the square root for $z\ge0$,
\begin{equation}
L'(z)=-\sqrt{2m\beta}\,z+2\sqrt{\sqrt{2m\beta}\,L(z)}.
\label{eqL3}
\end{equation}
(We put a plus sign in front of the square root by analogy with the $p<2$
case.)

Note already that \eqref{predicA} with \eqref{ansatz3} allows one to
recover the results from Berestycki \textit{et al.}~\cite{berestycki_et_al:as_growth_x2_bbm}:
\begin{equation}
\text{For $p=2$ and $m=1$,}
\lim_{T\to\infty}\frac1T\log\log|N(T)|=2\sqrt{2\beta}\quad\text{almost surely},
\label{N p=2}
\end{equation}
and from Harris and Harris~\cite{harris_harris:inhom_breeding}:
\begin{equation}
\text{For $p=2$ and $m=1$,}
\lim_{T\to\infty}\frac1T\log R_T =\sqrt{2\beta} \quad\text{almost surely}.
\label{R p=2}
\end{equation}
(Although the two papers~\cite{berestycki_et_al:as_growth_x2_bbm}  and~\cite{harris_harris:inhom_breeding} only concerned the binary branching ($m=1$) case, their results \eqref{N p=2} and~\eqref{R p=2} could easily be extended to the more general branching process of the present paper, with an arbitrary value of $m$.)

Going further, \eqref{eqL3} can be solved by making the change
of variable $L(z)=\sqrt{2m\beta}\,z^2\phi(z)^2$ with $\phi(z)\ge0$. One
gets for $z\ge0$
\begin{equation}
2z\phi^2(z)+2z^2\phi(z)\phi'(z)=-z+2z\phi(z)
\end{equation}
For $z>0$ the variables can be separated
\begin{equation}
\frac{2\phi(z)\phi'(z)}{2\phi^2(z)-2\phi(z)+1}=-\frac1z,
\end{equation}
and, after integration of both sides and simplification, one gets
an implicit form for $L(z)$:
\begin{equation}
\begin{cases}
\displaystyle L(z)=\sqrt{2m\beta}\,z^2\phi^2\\
\displaystyle
z=C\frac{\exp\left[\arctan(1-2\phi)\right]}{\sqrt{2\phi^2-2\phi+1}}\qquad\text{with $\phi\ge0$}
\end{cases}
\end{equation}
where $C$ is an integration constant.
By taking $\phi\to\infty$ one
obtains $L(0)$. The rescaled position $\bar z$ of the right-most particle is
when $L(\bar z)=0$ or $\phi=0$. From \eqref{eqL2}, the optimal position $\hat z$ where $L$ is maximal is such that $L(\hat z)=\sqrt{2m\beta}\,{\hat z}^2/4$ or $\phi=1/2$. This leads to
\begin{equation}
L(0)=\sqrt{2m\beta}\,\frac{C^2}2 e^{-\pi},\quad
\bar z=C e^{\frac\pi4},\quad
\hat z = C\sqrt2,\quad
L(\hat z)=\sqrt{2m\beta}\,\frac{C^2}2.
\label{p=2}
\end{equation}

\begin{rmk}
We have no theory for the value of the integration constant $C$; in fact we
believe that $C$ is realisation-dependent. This method allows us however to
make some conjectures on the values of several ratios such as $\bar z/\hat
z = $(the position of the right-most)/(the position where the density of
particles is the highest), or $L(\hat z)/\bar z^2=\log
|N_T|/(\text{position of the right-most})^2$.
We have no demonstration for these conjectures. We simply observe that, for
instance, the value of the ratio $\bar z / \hat z$ computed in the $p<2$
case (see equation~\eqref{r(s)} and Theorem~\ref{constrained_thm})
converges as $p\to 2$ to the ratio predicted by~\eqref{p=2}.
\end{rmk}

\subsection{Explicit calculations for $p=1$}

It is interesting to note that for $p=1$ (as well as for the easier case $p=0$) the equations given by Theorems \ref{unconstrained_thm} and \ref{constrained_thm} can be solved explicitly.
\begin{itemize}
\item The rightmost particle satisfies $R_T \sim \frac1{2}{m\beta}T^2$
a.s.
\item The optimal path for expected growth with end-point $z\ge0$ is given by
\[h_z(s) = -\frac{1}{2}m\beta s^2 + (z+\frac{1}{2}m\beta)s\]
with growth rate
\[\nexp(z) = \frac{1}{24}
m^2\beta^2 +\frac{1}{2}m\beta z - \frac{1}{2}z^2.\]
\item The optimal end-point for expected growth is $\bestzexp = \frac{1}{2}m\beta$, giving the optimal path
\[h_{\bestzexp}(s) = -\frac{1}{2}m\beta s^2 + m\beta s\]
and total expected growth rate $\bestnexp = m^2\beta^2/6$.
\item $r(s)=\frac1{2}{m\beta}s^2$ and $\bar z = \frac12m\beta$.
\item The optimal path for almost sure growth with end-point $z\in[0,\bar z]$ is given by
\[g_z(s)=\left\{\begin{array}{ll} \frac{m\beta}{2}s^2  & \hbox{if } s \in[0,s_z]\\
-\frac{m\beta}{2}s^2 + 2m\beta s_z s -m\beta s_z^2 & \hbox{if } s\in(s_z,1]\end{array}\right.\]
where $s_z=1-\sqrt{\frac{1}{2}-\frac{z}{m\beta}}$.
The corresponding growth rate is
\[\nas(z) = m^2\beta^2\left(\frac{1}{2} - \frac{z}{m\beta} - \frac{1}{6}\left(2-\frac{4z}{m\beta}\right)^{3/2}\right).\]
\item The optimal end-point for almost sure growth is $\bestzas = \frac{m\beta}{4}$, giving the optimal path
\[g_{\bestzas}(s) = \left\{\begin{array}{ll} \frac{m\beta}{2}s^2 & \hbox{if } s\in\left[0,\frac{1}{2}\right]\\
  -\frac{m\beta}{2}s^2 + m\beta s  -
\frac{m\beta}{4} & \hbox{if } s\in\left(\frac{1}{2},1\right]\end{array}\right.\]
and total almost sure growth rate $\bestnas = m^2\beta^2/12$.
\end{itemize}
We note in particular that there is positive expected growth for all $|z| < m\beta\big(1/2+1/\sqrt{3}\big)$ despite almost sure extinction for all $|z|>m\beta/2$. Only for $p=0$ do the almost sure and expected growth match.

\begin{figure}[h!]
  \centering
\includegraphics[width=\textwidth]{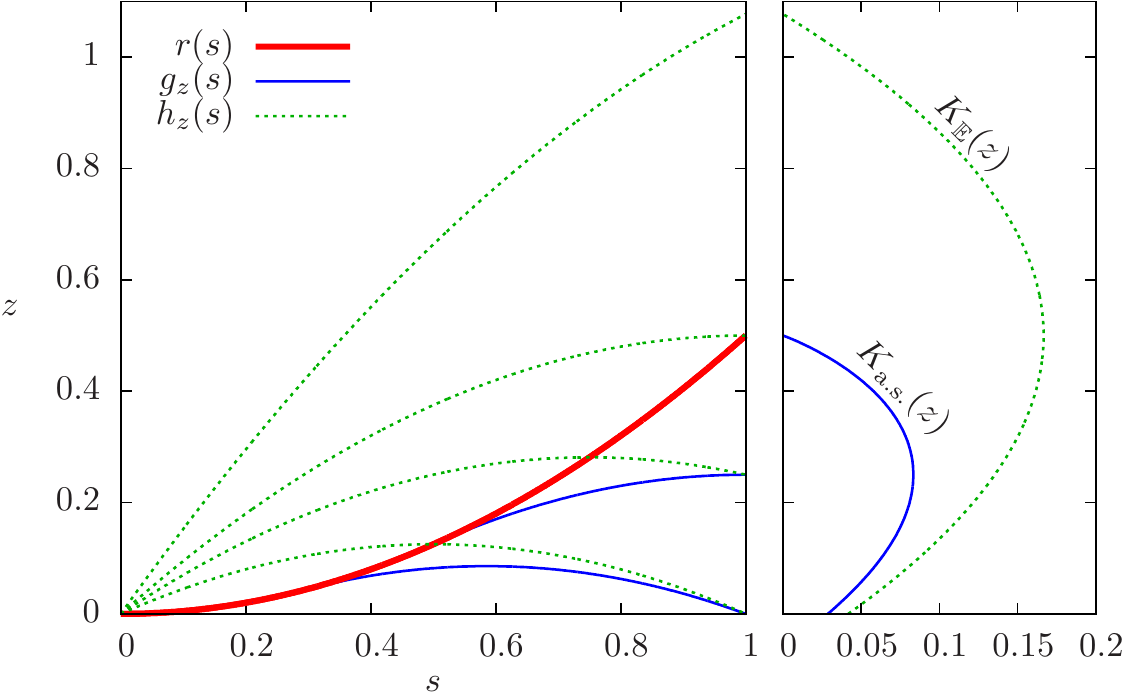}
  \caption{For $p=1$ and $m\beta=1$, the thick (red) line in the left-hand graph shows $r(s)$, the (rescaled) path taken by the right-most particle. The thinner lines show, for  various values of the endpoint $z$, the optimal paths $g_z(s)$ for almost sure growth (plain blue lines) and $h_z(s)$ for expected growth (dashed green lines). In the right-hand graph we show the profiles $\nas(z)$ and $\nexp(z)$ of the number of particles alive at each position $z$ for almost sure growth (plain blue) and expected growth (dashed green).}
\end{figure}

\subsection{Proof strategy}\label{heuristic_sec}

As already pointed out in the introduction, the results in expectation (Theorems \ref{expected_growth_paths_thm} and \ref{unconstrained_thm}) can be derived in a rather straightforward fashion from Shilder's large deviation theorem and the use of the so-called {\it many-to-one principle} (see Section \ref{expected_growth_sec}).
Some fairly standard large deviations techniques (using for example the exponential tightness of Brownian motion) can then be used to deduce the large deviations behaviour of the system seen in Theorem~\ref{large_devs_thm} (see Section \ref{large_devs_sec}).

For the almost sure growth along paths, however, we need something stronger. Using the many-to-one lemma
we construct in Section~\ref{spine_section} processes
that are non-negative martingales which count the numbers of particles
whose paths lie in certain sets.
We then use the fact that these martingales have almost surely finite
limits to obtain an almost sure upper bound on the number of particles
whose rescaled paths remain close to $f$. It is also quite usual, at least within the world of branching processes, that if an additive martingale --- like ours --- is uniformly integrable, then it has a strictly positive limit, giving us our almost sure lower bound. Again this is true in our case, although showing it is highly non-trivial --- a large part of the work for this article is spent in proving this lower bound.

We then set about proving the results concerning how many particles follow particular paths. In Section \ref{growth_paths_sec} we prove Theorem \ref{growth_paths_thm}, applying many of the results obtained in the previous two sections.

We move on in Section \ref{sec:optpaths} to derive the optimal paths seen in Theorem \ref{constrained_thm} and Theorem \ref{unconstrained_thm}, and study these paths further (in particular proving Theorem \ref{constrained_id}) in Section \ref{further_optpaths_sec}.



\section{A family of spine martingales}\label{spine_section}

\subsection{The spine setup}

A key idea in our proofs will be the use of certain additive martingales.
These martingales can be used to define changes of measure under which one
particle behaves differently than under the law $\P$ of our branching
particle system. The tools introduced in this way are extremely useful, and should be fairly intuitive. As they are now well-embedded in the branching process literature, we will leave out several proofs in this section, and refer the interested reader to Hardy and Harris' very general formulation in \cite{hardy_harris:spine_approach_applications}.

\vspace{3mm}

We first embellish our probability space by keeping track of some extra information about one particular infinite line of descent or \emph{spine}. This line of descent is defined as follows: our one initial particle is part of the spine; when this particle dies, we choose one of its children uniformly at random to become part of the spine. We continue in this manner: when the spine particle dies, we choose one of its children uniformly at random to become part of the spine. In this way at any time $t\geq0$ we have exactly one particle in $N(t)$ that is part of the spine. We refer to both this particle and its position with the label $\xi_t$; this is a slight abuse of notation, but it should always be clear from the context which meaning is intended. The spatial motion of the spine, $(\xi_t)_{t\geq0}$, is a standard Brownian motion.

The resulting probability measure we denote by $\Pt$, and we find need for four different filtrations to encode differing amounts of this new information:
\begin{itemize}
\item{$\Fg_t$ contains all the information about the original system up to time $t$. However, it does not know which particle is the spine at any point. Thus it is simply the natural filtration of the branching Brownian motion.}
\item{$\Ft_t$ contains all the information about both the BBM and the spine up to time $t$.}
\item{$\Gt_t$ contains all the information about the spine up to time $t$, including the birth times of other particles along its path and how many children are born at each of these times; it does not know anything about the rest of the tree.}
\item{$\Gg_t$ contains just the spatial information about the spine up to time $t$; it does not know anything about the rest of the tree.}
\end{itemize}
We note that $\Fg_t \subset \Ft_t$ and $\Gg_t\subset\Gt_t\subset\Ft_t$, and also that $\Pt$ is an extension of $\Pb$ in that $\Pb = \Pt |_{\Fg_\infty}$.

\begin{lem}[Many-to-one lemma]
If $g(t)$ is  $\Ft_t$-measurable it can be written in the form
\[g(t) = \sum_{u\in N(t)} g_u(t)\ind_{\{\xi_t = u\}}\]
where each $g_u(t)$ is $\Fg_t$-measurable, and then
\[\Eb\left[\sum_{u\in N(t)} g_u(t)\right] = \tilde\Eb[e^{m\beta\int_0^t|\xi_s|^p \diffd s}g(t)].\]
\end{lem}

This lemma is extremely useful as it allows us to reduce questions about the entire population down to calculations involving just one standard Brownian motion --- the spine. A proof of a more general version of this lemma may be found in \cite{hardy_harris:spine_approach_applications}.

\subsection{Martingales and changes of measure}
For $p\in[0,2)$, $f\in C[0,1]$, $\theta\in[0,1]$ and $\varepsilon>0$, let
\[q:= \frac{2}{2-p} \in[1,\infty)\]
and define
\[N_T(f,\varepsilon,\theta) := \left\{u\in N(\theta T) : |X_u(t)-T^q f(t/T)|<\varepsilon T^q \hs \forall t\in[0,\theta T]\right\}\]
so that $N_T(f,\varepsilon,\theta) = N_T(B(f,\varepsilon),\theta)$, see
\eqref{NTD}, where
\[B(f,\varepsilon):=\left\{g\in C[0,1]:||f-g||_\infty <\varepsilon\right\}.\]
We look for martingales associated with these sets. For convenience, in this section we use the shorthand
\[N_T(t):= N_T(f,\varepsilon,t/T)\]
and
\[C_T(x,t):= \cos\left(\frac{\pi}{2\varepsilon T^q}(x-T^q f(t/T))\right).\]

The following Lemma is adapted from Lemma 6 in  \cite{harris_roberts:unscaled_paths}.
\begin{lem}
If $f\in C^2[0,1]$ then the process
\[V_T(t) := e^{\pi^2 t / (8\varepsilon^2 T^{2q}) + T^{q-1}\int_0^t f'(s/T)
\diffd\xi_s - \frac{1}{2}T^{2q-2}\int_0^{t} f'(s/T)^2 \diffd s}C_T(\xi_t,t),\hs\hsl t\in[0,T]\]
is a $\Gg_t$-local martingale under $\Pt$.
\end{lem}

\begin{proof}
Since the motion of the spine is simply a standard Brownian motion under $\Pt$, this is easily checked by applying It\^{o}'s formula (the sufficient conditions of, for example, Lawler \cite{lawler:intro_stochastic_processes} tell us that if $f\in C^2[0,1]$ then $V_T$ is sufficiently smooth for It\^o's formula to hold). See \cite{harris_roberts:unscaled_paths}, Lemmas 5 and 6, for the calculations.
\end{proof}

By stopping the process $(V_T(t), t\in[0,T])$ at the first exit time of the Brownian motion from the tube $\{(x,t): |T^q f(t/T)-x|< \varepsilon T^q\}$, we obtain also that
\[\zeta_T(t) := V_T(t) \ind_{\{|T^q f(s/T)-\xi_s| < \varepsilon T^q \hsl \forall s\leq t\}}, \hs\hs t\in[0,T]\]
is a non-negative $\Gg_t$-local martingale, and since its size is then clearly constrained it must in fact be a $\Gg_t$-martingale. As in \cite{hardy_harris:spine_approach_applications}, we may build from $\zeta_T$ a collection of $\Ft_t$-martingales $\tilde\zeta_T$ given by
\[\tilde\zeta_T(t) := \prod_{v<\xi_t} (1+A_v) e^{-m\beta\int_0^t |\xi_s|^p \diffd s}\zeta_T(t), \hs\hs t\in[0,T]\]
where we write $\{v<\xi_t\}$ for the set of strict ancestors of the spine particle at time $t$. When we project $\tilde\zeta_T(t)$ back onto $\Fg_t$ we get a new set of mean-one $\Fg_t$-martingales $(Z_T(t),t\geq0)$. These processes $Z_T$ are the main objects of interest in this section, and can be expressed for $t\in[0,T]$ as the sum
\[Z_T(t) = \sum_{u\in N_T(t)} \zeta_T^{(u)}(t)e^{-m\beta\int_0^t |X_u(s)|^p \diffd s}= \sum_{u\in N_T(t)} V_T^{(u)}(t)e^{-m\beta\int_0^t |X_u(s)|^p \diffd s}\]
where $\zeta_T^{(u)}(t)$ and  $V_T^{(u)}(t)$ are simply $\zeta_T(t)$ and $V_T(t)$ with the path $X_u(s)$ replacing $\xi(s),$ i.e.
\[V_T^{(u)}(t) := e^{\pi^2 t / (8\varepsilon^2 T^{2q}) + T^{q-1}\int_0^t
f'(s/T) \diffd X_u(s) - \frac{1}{2}T^{2q-2}\int_0^{t} f'(s/T)^2 \diffd s} C_T(X_u(t),t)  .\]

We now proceed to show that the martingales $Z_T$ are close to
\begin{equation}\label{almost_mg}
\sum_{u\in N(t)} \ind_{\{u \hbox{\scriptsize{ is close to }} f\}}
e^{T^{2q-2}\int_0^t \big[\frac{1}{2}f'(s/T)^2 -m\beta |f(s/T)|^p
\big]\diffd s}
\end{equation}
and that they have the properties that we discussed in Section \ref{heuristic_sec} --- specifically, we aim to show that for certain $f$ the martingales $Z_T$ are uniformly integrable and thus cannot be too small.
This is the key step to counting particles whose rescaled paths stay close
to $f$.

We define new measures, $\Qt_T$, via
\[\Qt_T|_{\Ft_t} := \tilde\zeta_T(t)\Pt|_{\Ft_t}\]
for $t\in[0,T]$ --- and note that
\[\Qt_T|_{\Fg_t} = Z_T(t)\Pt|_{\Fg_t} \hs\hs \hbox{and} \hs\hs \Qt_T|_{\Gg_t} = \zeta_T(t)\Pt|_{\Gg_t}.\]

\begin{lem}\label{change_of_meas_lem}
Under $\Qt_T$ the spine $(\xi_t,t\in[0,T])$ moves as a Brownian motion with drift
\[T^{q-1}f'(t/T) - \frac{\pi}{2\varepsilon T^q} \tan\left(\frac{\pi}{2\varepsilon T^q}(x - T^q f(t/T))\right)\]
when at position $x$ at time $t$; in particular,
\[|\xi_t - T^q f(t/T)|\leq\varepsilon T^q \hs \forall t\leq T.\]
Each particle $u$ in the spine dies at an accelerated rate $(m+1)\beta|x|^p$ when in position $x$, to be replaced by a random number $A_u+1$ of offspring where $A_u$ is taken from the \emph{size-biased} distribution relative to $A$, given by $\Qt_T(A_u = k) = (m+1)^{-1}(k+1)P(A=k)$ (note that this distribution does not depend on $T$). All non-spine particles, once born, behave \emph{exactly as they would under $\Pb$}: they move like independent standard Brownian motions, die at the normal rate $\beta |x|^p$, and give birth to a number of particles that is distributed like $1+A$.
\end{lem}

\begin{proof}
A proof of this result can be found in \cite{hardy_harris:spine_approach_applications}. We will not use the precise drift of the spine except for the fact that it remains within the tube: to see this note that since the event is $\Gg_T$-measurable,
\[\Qt_T(\exists t\leq T : |\xi_t - T^q f(t/T)|>\varepsilon T^q) = \tilde\Eb[\zeta_T(T)\ind_{\{\exists t\leq T : |\xi_t - T^q f(t/T)|>\varepsilon T^q\}}] = 0\]
by the definition of $\zeta_T(T)$.
\end{proof}

Another important tool is the spine decomposition, which will allow us to bound the growth of the martingales $Z_T$ via one-particle calculations in a more delicate way than is possible via the many-to-one lemma. A proof of a more general version of the spine decomposition may be found in \cite{hardy_harris:spine_approach_applications}.

\begin{thm}[Spine decomposition]
$\Qt_T$-almost surely,
\[\Qt_T[Z_T(t)|\tilde\Gg_T] = \sum_{u<\xi_t} A_u V_T(S_u)e^{-m\beta\int_0^{S_u}|\xi_s|^p \diffd s} + V_T(t)e^{-m\beta\int_0^{t}|\xi_s|^p \diffd s}\]
where $\{u<\xi_t\}$ is the set of ancestors of the spine particle at time $t$, and $S_u$ denotes the time at which particle $u$ died and split into $1+A_u$ new particles.
\end{thm}

As we have already mentioned, the main aim of introducing these martingales is to give us a lower bound on the number of particles in $N_T(t)$. To do this we must bound the size of each of the terms in the sum. The following lemma is a simple bound for the Girsanov part of the martingale, using integration by parts.

\begin{lem}
\label{scaled_int_lem}
If $f\in C^2[0,1]$ and $f(0)=0$ then for any $u\in N_T(t)$, almost surely under both $\Pt$ and $\Qt_T$ we have
\begin{multline*}
\left|T^{q-1}\int_0^t f'(s/T) \diffd X_u(s) - T^{2q-2}\int_0^{t} f'(s/T)^2 \diffd s\right|\\
\leq 2\varepsilon T^{2q-2}\int_0^{t} |f''(s/T)| \diffd s + \varepsilon T^{2q-1}|f'(0)|.
\end{multline*}
\end{lem}

\begin{proof}
From the integration by parts formula for It\^o calculus (since for any particle $u\in N(t)$, $(X_u(s), 0\leq s\leq t)$ is a Brownian motion under $\Pt$) we know that for any $g\in C^2[0,1]$, under $\Pt$,
\[g'(t)X_u(t) = \int_0^t g''(s)X_u(s) \diffd s + \int_0^t g'(s)\diffd X_u(s).\]
From ordinary integration by parts, if $g(0)=0$,
\[\int_0^t g'(s)^2 \diffd s = g'(t)g(t) - \int_0^t g(s)g''(s) \diffd s.\]
Now set $g(t)=T^q f(t/T)$ for $t\in[0,T]$. We note that if $u\in N_T(t)$ then $|X_u(s) - g(s)|< \varepsilon T^q$
for all $s\leq t$. Thus
\begin{align*}
&\biggl|T^{q-1}\int_0^t f'(s/T) \diffd X_u(s) - T^{2q-2}\int_0^{t} f'(s/T)^2 \diffd s\hspace{0.5mm}\biggr|\\
&= \biggl|\int_0^t g'(s) \diffd X_u(s) - \int_0^t g'(s)^2 \diffd s\hspace{0.5mm}\biggr|\\
&\leq \biggl|\hspace{0.5mm}g'(t)(X_u(t)-g(t)) - \int_0^t g''(s)(X_u(s)-g(s))\diffd s\hspace{0.5mm}\biggr|\\
& \leq  |g'(t)-g'(0)|\times |X_u(t)-g(t)| +|g'(0)|\times|X_u(t)-g(t)|\\
&\hspace{20mm}+ \int_0^t |g''(s)|\times|X_u(s)-g(s)|\diffd s\\
&\leq 2\varepsilon T^q\int_0^{t} |g''(s)| \diffd s + \varepsilon T^q |g'(0)|\\
&= 2\varepsilon T^{2q-2}\int_0^{t} |f''(s/T)|\diffd s + \varepsilon T^{2q-1}|f'(0)|
\end{align*}
almost surely under $\Pt$ and, since $\Qt_T \ll \Pt$ (on $\Ft_T$), almost surely under $\Qt_T$.
\end{proof}

The next lemma continues along the same theme, controlling the terms in $Z_T$ so that eventually we will be able to give a lower bound on $Z_T$ and then use this to give a lower bound on $N_T$ by showing that $Z_T$ looks something like (\ref{almost_mg}).

\begin{lem}\label{xp_lem}
For any $u\in N_T(t)$,
\[T^{2q-2} \inf_{g\in B(f,\varepsilon)} \int_0^{t} |g(s/T)|^p \diffd s \leq \int_0^t|X_u(s)|^p \diffd s \leq T^{2q-2} \sup_{g\in B(f,\varepsilon)} \int_0^{t} |g(s/T)|^p \diffd s.\]
\end{lem}

\begin{proof}
This follows immediately from the fact that if $u\in N_T(t)$ then (by definition) there exists $g\in B(f,\varepsilon)$ such that $X_u(s) = T^q g(s/T)$ for all $s\leq t$.
\end{proof}

We are now ready to start putting together the bounds that we have given. Combining Lemmas \ref{scaled_int_lem} and \ref{xp_lem} with the spine decomposition we obtain the following.

\begin{lem}\label{spine_domination}
If $f\in C^2[0,1]$, $f(0)=0$, $f'(0)=0$ and $m\beta\int_0^\phi |f(s)|^p \diffd s > \frac{1}{2}\int_0^\phi f'(s)^2 \diffd s$ for all $\phi\in(0,\theta]$, then for small enough $\varepsilon>0$ and any $T>0$ and $t\leq\theta T$, there exists $\eta>0$ such that
\[\Qt_T[Z_T(t)|\Gt_T] \leq \sum_{u<\xi_t} A_u e^{\pi^2/(8\varepsilon^2 T^{2q-1}) - \eta m \beta \int_0^{S_u}|\xi_s|^p \diffd s} + e^{\pi^2/(8\varepsilon^2 T^{2q-1}) - \eta m\beta\int_0^t|\xi_s|^p \diffd s}\]
$\Qt_T$-almost surely.
\end{lem}

\begin{proof}
Recall that under $\Qt_T$ the spine is in $N_T(t)$ for all $t\leq T$. Thus by Lemmas \ref{scaled_int_lem} and \ref{xp_lem}, since $f'(0)=0$, for any $\eta\in(0,1)$,
\begin{multline*}
-m\beta\int_0^t |\xi_s|^p \diffd s + T^{q-1}\int_0^t f'(s/T) \diffd\xi_s - \frac{1}{2}T^{2q-2}\int_0^{t} f'(s/T)^2 \diffd s\\
\leq -\eta m\beta\int_0^t |\xi_s|^p \diffd s -(1-\eta)m\beta T^{2q-2}\inf_{g\in B(f,\varepsilon)} \int_0^{t} |g(s/T)|^p \diffd s\\
+ \frac{1}{2}T^{2q-2}\int_0^{t} f'(s/T)^2 \diffd s + 2\varepsilon T^{2q-2}\int_0^{t} |f''(s/T)|\diffd s
\end{multline*}
for all $t\leq T$. Then, since $m\beta\int_0^\phi |f(s)|^p \diffd s > \frac{1}{2}\int_0^\phi f'(s)^2 \diffd s$ for all $\phi\in(0,\theta]$, we may choose $\varepsilon>0$  and $\eta>0$ small enough such that 
\begin{multline*}
-(1-\eta)m\beta T^{2q-2}\inf_{g\in B(f,\varepsilon)} \int_0^{t} |g(s/T)|^p \diffd s\\
 + \frac{1}{2}T^{2q-2}\int_0^{t} f'(s/T)^2 \diffd s + 2\varepsilon T^{2q-2}\int_0^{t} |f''(s/T)|\diffd s \leq 0
\end{multline*}
for all $t\in[0,\theta T]$. Plugging this into the spine decomposition, we get
\[\Qt_T[Z_T(t)|\Gt_T] \leq \sum_{u<\xi_t} A_u e^{\pi^2/8\varepsilon^2 T^{2q-1} - \eta  m \beta \int_0^{S_u}|\xi_s|^p \diffd s} + e^{\pi^2/8\varepsilon^2 T^{2q-1} - \eta m \beta\int_0^t|\xi_s|^p \diffd s}. \qedhere\]
\end{proof}

We are in a position now to complete one of our two initial aims, which was to show that the martingales $Z_T$ are uniformly integrable. This will be used in Section \ref{lower_bound_section} to show that $Z_T$ cannot be too small.

\begin{prop}\label{UI}
Take $f\in C^2[0,1]$ and $\theta \in [0,1].$ If $f(0)=0$, $f'(0)=0$, and $m\beta\int_0^\phi |f(s)|^p \diffd s > \frac{1}{2}\int_0^\phi f'(s)^2 \diffd s$ for all $\phi\in(0,\theta]$, then for small enough $\varepsilon>0$ the set
$\{Z_T(t) : T\geq 1, \hsl t\leq\theta T\}$ is uniformly integrable under $\Pb$.
\end{prop}

\begin{proof}
Fix $\delta>0$. We first claim that there exists $K$ such that
\[\sup_{\substack{T\geq1\\t\leq\theta T}}\Qt_T(\Qt_T[Z_T(t)|\Gt_T] > K) < \delta/2.\]
To see this, take an auxiliary probability space with probability measure
$Q$, and on this space consider a sequence $A_1, A_2, \ldots$ of random
variables with the same (size-biased) distribution as $A$ under $\Qt_T$
(there is no dependence on $T$) and a sequence $e_1, e_2, \ldots$ of random
variables that are exponentially distributed with parameter $\beta(m+1)$;
then set $\mathcal{S}_n = e_1 + \cdots + e_n$ (so that the random variable $\mathcal{S}_n$ has
the same distribution as $\int_0^{S_u} |\xi_s|^p \diffd s$, where $S_u$ is
the time of the $n$th splitting event along the spine under $\Qt_T$). By Lemma~\ref{spine_domination} we have (since $2q-1 \geq 1$)
\[\sup_{\substack{T\geq1\\t\in[1,\theta T]}}\Qt_T(\Qt_T[Z_T(t)|\Gt_T] > K) \leq Q\left(\sum_{j=1}^\infty A_j e^{\pi^2/8\varepsilon^2 - \eta \mathcal{S}_j} + e^{\pi^2/8\varepsilon^2} > K\right).\]
Hence our claim holds if the random variable
\[\sum_{j=1}^\infty A_j e^{- \eta \mathcal{S}_j}\]
can be shown to be $Q$-almost surely finite. Now for any $\gamma\in(0,1)$,
\begin{align*}
Q\left(\sum_n A_n e^{-\eta \mathcal{S}_n} = \infty\right) &\leq Q(A_n e^{-\eta \mathcal{S}_n} > \gamma^n \hbox{ infinitely often})\\
&\leq Q\left(\frac{\log A_n}{n} > \log \gamma + \frac{\eta \mathcal{S}_n}{n} \hbox{ infinitely often}\right).
\end{align*}
By the strong law of large numbers, $\mathcal{S}_n/n \to 1/\beta(m+1)$ almost surely under $Q$; so if $\gamma\in(\exp(-\eta/\beta(m+1)),1)$ then the quantity above is no larger than
\[Q\left(\limsup_{n\to\infty} \frac{\log A_n}{n} > 0\right).\]
But this quantity is zero by Borel-Cantelli: for any $T$,
\begin{align*}
\sum_n Q\left(\frac{\log A_n}{n}>\varepsilon\right) &= \sum_n \Qt_T(\log A > \varepsilon n )\\
&\leq \int_0^\infty \Qt_T(\log A \geq \varepsilon x) dx\\
&= \Qt_T\left[\frac{\log A}{\varepsilon}\right],
\end{align*}
which is finite for any $\varepsilon>0$ since (by direct calculation from the distribution of $A$ under $\Qt_T$ given in Lemma \ref{change_of_meas_lem}) $\Qt_T[\log A] = \Pt[A\log A]<\infty$. Thus our claim holds.

Now choose $M>0$ such that $1/M < \delta/2$; then for $K$ chosen as above, and any $T\geq1$, $t\leq\theta T$,
\begin{align*}
\Qt_T(Z_T(t) > MK) &\leq \Qt_T( Z_T(t) > MK, \hsl \Qt_T[Z_T(t)|\Gt_T]\leq K )\\
&\hspace{45mm} + \Qt_T(\Qt_T[Z_T(t)|\Gt_T] > K )\\
&\leq \Qt_T\left[\frac{Z_T(t)}{MK}\ind_{\{\Qt_T[Z_T(t)|\Gt_T] \leq K\}}\right] + \delta/2\\
&= \Qt_T\left[\frac{\Qt_T[Z_T(t)|\Gt_T]}{MK}\ind_{\{\Qt_T[Z_T(t)|\Gt_T] \leq K\}}\right] + \delta/2\\
&\leq 1/M + \delta/2 \leq \delta.
\end{align*}
Thus, setting $K'=MK$, for any $T\geq1$, $t\leq\theta T$,
\[\Pb[Z_T(t)\ind_{\{Z_T(t)>K'\}}] = \Qt_T(Z_T(t)>K') \leq \delta.\]
Since $\delta>0$ was arbitrary, the proof is complete.
\end{proof}

Finally we show that $Z_T$ is close to (\ref{almost_mg}).

\begin{lem}\label{int_by_parts_lem}
For any $\delta>0$, if $f\in C^2[0,1]$, $f(0)=0$ and $\varepsilon$ is small enough then
\[Z_T(\theta T) \leq |N_T(f,\varepsilon,\theta)|e^{\frac{\pi^2\theta}{8\varepsilon^2 T^{2q}} - m\beta T^{2q-1}\int_0^\theta |f(\sigma)|^p \diffd \sigma + \frac{1}{2}T^{2q-1}\int_0^\theta f'(\sigma)^2 \diffd \sigma + \delta T^{2q-1}}.\]
\end{lem}

\begin{proof}
Simply plugging the results of Lemmas \ref{scaled_int_lem} and \ref{xp_lem} into the definition of $Z_T(\theta T)$ gives the desired inequality.
\end{proof}

We note here that, in fact, a similar bound can easily be given in the opposite direction, so that $|N_T(f,\varepsilon/2,\theta)|$ is dominated by $Z^T(\theta T)$ multiplied by some deterministic function of $T$. We will not need this bound, but it is interesting to note that the study of the martingales $Z_T$ is in a sense equivalent to the study of the number of particles $N_T$.

\section{Proof of Theorem \ref{expected_growth_paths_thm}}\label{expected_growth_sec}
We first rule out the possibility of any particles following unusual paths, which allows us to restrict our attention to compact sets, and hence small balls about sensible paths. In this section we go back to our original notation convention in which small letters such as $s$ and $t$ as well as $\theta$ are used for scaled time parameters varying in $[0,1]$ and capital letters are reserved for the non-scaled time.

\begin{lem}\label{extreme_lem}
Fix $\theta\in[0,1]$. For $N\in\mathbb{N}$, let
\[F_N:=\left\{f\in C[0,1] : \exists n\geq N, \hsl s,t\in[0,\theta] \hbox{ with } |t-s|\leq \frac{1}{n^2}, \hsl |f(t)-f(s)|>\frac{1}{\sqrt n}\right\}.\]
For any $\eta>0$ we may choose $N\in\mathbb{N}$ such that for all large $T$
\[\P(N_T(F_N,\theta)\neq\emptyset) \leq \E[N_T(F_N,\theta)] \leq \exp(-\eta T^{2q-1})\]
and
\[\limsup_{T\to\infty}\frac{1}{T^{2q-1}}\log |N_T(F_N,\theta)| = -\infty\]
almost surely.
\end{lem}

\begin{proof}
Fix $T\geq S\geq0$; then for any $U\in[S,T]$,
\begin{align*}
\{\xi_U\in N_U(F_N,\theta)\} &= \left\{\exists n\geq N,\hsl s,t\in[0,\theta] : |t-s|\leq \frac{1}{n^2},\hsl \left|\frac{\xi_{sU} - \xi_{tU}}{U^q}\right| > \frac{1}{\sqrt n}\right\}\\
&\subset \left\{\exists n\geq N,\hsl s,t\in[0,\theta] : |t-s|\leq \frac{1}{n^2},\hsl \left|\frac{\xi_{sT} - \xi_{tT}}{S^q}\right| > \frac{1}{\sqrt n}\right\}.
\end{align*}
Since the right-hand side does not depend on $U$, we deduce that
\begin{multline*}
\{\exists U\in[S,T] : \xi_U\in N_U(F_N,\theta)\} \\ \subset \left\{\exists n\geq N,\hsl s,t \in[0,\theta] : |t-s|\leq \frac{1}{n^2},\hsl \left|\frac{\xi_{sT} - \xi_{tT}}{S^q}\right| > \frac{1}{\sqrt n}\right\}.\end{multline*}
Now, for $s\in[0,\theta]$, define $\pi(n,s):=\lfloor 2n^2 s \rfloor
/(2n^2)$. Suppose we have a continuous function $f$ such that
\mbox{$\sup_{s\in[0,\theta]}|f(s) - f(\pi(n,s))|\leq 1/(4\sqrt n)$}. If $s,t\in[0,\theta]$ satisfy $|t-s|\leq 1/n^2$, then
\begin{align*}
&|f(t) - f(s)|\\
&\leq |f(t) - f(\pi(n,t))| + |f(s) - f(\pi(n,s))| + |f(\pi(n,s)) - f(\pi(n,t))|\\
&\leq \frac{1}{4\sqrt n} + \frac{1}{4\sqrt n} + \frac{2}{4\sqrt n} = \frac{1}{\sqrt n}.
\end{align*}
Thus
\[\{\exists U\in[S,T] : \xi_U \in N_U(F_N,\theta)\}\subset\left\{\exists n\geq N,\hsl s\leq\theta : \left|\frac{\xi_{sT}-\xi_{\pi(n,s)T}}{S^q}\right|>\frac{1}{4\sqrt n}\right\}.\]
Standard properties of Brownian motion now give us that
\begin{align*}
\Pb(\exists U\in[S,T] : \xi_U\in N_U(F_N,\theta)) &\leq \Pb\left(\exists n\geq N,\hsl s\leq\theta : |\xi_{sT}-\xi_{\pi(n,s)T}|>S^q/4\sqrt n\right)\\
&\leq \sum_{n\geq N} 2n^2
\Pb\left(\sup_{s\in[0,1/(2n^2)]}\left|\xi_{sT}\right|> S^q/4\sqrt n \right)\\
&\leq \sum_{n\geq N} \frac{8\sqrt{n^3 T}}{S^q\sqrt{\pi}}\exp\left(-\frac{S^{2q} n}{16T}\right).
\end{align*}
Taking $S=j$ and $T=j+1$, we note that for large $N$,
\[\sum_{n\geq N} \frac{8\sqrt{n^3 T}}{S\sqrt{\pi}}\exp\left(-\frac{S^{2q} n}{16T}\right) \leq \sum_{n\geq N} \exp\left(-\frac{j^{2q-1} n}{32}\right).\]
Now, for any $M>0$,
\begin{align*}
&\Pb\left(\sup_{T\in[j,j+1]} |N_T(F_N,\theta)|\geq 1\right) \leq \Eb\left[\sup_{T\in[j,j+1]} |N_T(F_N,\theta)|\right]\\
&\leq \Eb\left[\sum_{u\in N(j+1)}\ind_{\{\exists T \in[j,j+1] : u\in N_T(F_N,\theta)\}}\right] = \Eb\left[e^{m\beta\int_0^{j+1} |\xi_s|^p \diffd s} \ind_{\{\exists T \in[j,j+1] : \xi_T\in N_T(F_N,\theta)\}}\right]\\
&\leq \Eb\left[e^{m\beta\int_0^{j+1} |\xi_s|^p \diffd s} \ind_{\{\exists T \in[j,j+1] : \xi_T\in N_t(F_N,\theta)\}}\ind_{\{\sup_{S\leq j+1}|\xi_S|\leq M(j+1)^q
\}}\right]\\
&\hspace{12mm} + \Eb\left[e^{m\beta\int_0^{j+1} |\xi_s|^p \diffd s} \ind_{\{\sup_{S\leq j+1}|\xi_S| > M(j+1)^q\}}\right]\\
&\leq e^{m\beta M^p (j+1)^{pq+1}} \Pb(\exists T\in[j,j+1] : \xi_T \in N_T(F_N,\theta))\\
&\hspace{12mm} + \sum_{k\geq 1} \Eb\left[e^{m\beta\int_0^{j+1} |\xi_s|^p \diffd s} \ind_{\{\sup_{S\leq j+1}|\xi_T|\in[kM(j+1)^q,(k+1)M(j+1)^q]\}}\right]\\
&\leq e^{m\beta M^p (j+1)^{pq+1}} \Pb(\exists T\in[j,j+1] : \xi_T \in N_T(F_N,\theta))\\
&\hspace{12mm} + \sum_{k\geq 1} e^{m\beta (j+1)^{2q-1}(k+1)^p M^p} \Pb(\sup_{S\leq j+1}|\xi_S| \in [kM(j+1)^q, (k+1)M(j+1)^q])\\
&\leq e^{m\beta M^p (j+1)^{pq+1}} \Pb(\exists T\in[j,j+1] : \xi_T \in N_T(F_N,\theta))\\
&\hspace{12mm} + 4\sum_{k\geq 1} \frac{1}{\sqrt{2\pi}(j+1)}e^{m\beta (j+1)^{2q-1}(k+1)^p M^p - k^2M^2 (j+1)^{2q-1}/2}.
\end{align*}
Both of the terms in the right-hand side can be made exponentially small in $j$ by
choosing $M$, and then $N$, sufficiently large (for the first, see our calculations earlier in the proof). This establishes the first part of the lemma, and by Borel-Cantelli we have that for large enough $N$
\[\Pb(\limsup_{j\to\infty} \sup_{T\in[j,j+1]} |N_T(F_N,\theta)| \geq 1)=0\]
and since $|N_T(F_N,\theta)|$ is integer-valued,
\[\limsup_{T\to\infty}\frac{1}{T^{2q-1}}\log|N_T(F_N,\theta)| = -\infty\]
almost surely.
\end{proof}

We now check that we can cover our sets in a suitable way.

\begin{lem}\label{totally_bdd_lem}
Let
\[C_0[0,1]:=\{f\in C[0,1] : f(0)=0\}.\]
For each $N\in\mathbb{N}$, the set $C_0[0,1]\setminus F_N$ is totally
bounded under $\|\cdot\|_\infty$ (that is, it may be covered by finitely
many open balls of arbitrarily small radius).
\end{lem}

\begin{proof}
Given $\varepsilon>0$, choose $n$ such that \mbox{$n\geq N\vee 1/\varepsilon^2$}. For any \mbox{$f\in C_0[0,1]\setminus F_N$}, if \mbox{$|u-s|<1/n^2$} then \mbox{$|f(u)-f(s)|\leq 1/\sqrt n \leq \varepsilon$}. Thus \mbox{$C_0[0,\theta]\setminus F_N$} is equicontinuous (and, since each function must start from 0, uniformly bounded) and we may apply the Arzel\`a-Ascoli theorem to say that $C_0[0,1]\setminus F_N$ is relatively compact, which is equivalent to totally bounded since $(C[0,1], \|\cdot\|_\infty)$ is a complete metric space.
\end{proof}

\begin{lem}\label{semicont_lem}
For $D\subset C_0[0,1]$ define
\[D^\varepsilon := \{f\in C_0[0,1] : \exists g\in D \hbox{ with } \|g-f\|_\infty\leq\varepsilon\}.\]
For any $\delta>0$ there exists $\varepsilon>0$ such that
\[\sup_{f\in D^\varepsilon} K(f,\theta) \leq \sup_{f\in D} K(f,\theta) + \delta.\]
\end{lem}

\begin{proof}
For each $\theta$, $f \mapsto \int_0^\theta f'(s)^2 \diffd s$ is a lower semicontinuous function on $C_0[0,\theta]$: we refer to Section 5.2 of \cite{dembo_zeitouni:large_devs} but it is possible to give a direct proof. Thus $f \mapsto m\beta \int_0^\theta |f(s)|^p \diffd s - \frac{1}{2}\int_0^\theta f'(s)^2 \diffd s$ is upper semicontinuous. Now, by Jensen's inequality, for any $f\in C_0[0,\theta]\cap H_1$ and any $s,t\in[0,\theta]$, $s<t$,
\[\frac{1}{t-s}\int_s^t f'(u)^2 \diffd u \geq \left(\frac{1}{t-s}\int_s^t
f'(u) \diffd u\right)^2 = \left(\frac{f(t)-f(s)}{t-s}\right)^2\]
so that
\begin{equation}\label{jensen}
(f(t)-f(s))^2 \leq (t-s)\int_s^t f'(u)^2 \diffd u.
\end{equation}
There exists $t\in[0,\theta]$ such that $|f(t)|^p \geq \frac{1}{\theta}\int_0^\theta |f(s)|^p \diffd s$, so by (\ref{jensen}) (taking $s=0$)
\[\int_0^\theta f'(u)^2 \diffd u \geq \int_0^t f'(u)^2 \diffd u \geq \frac{\left(\int_0^\theta |f(s)|^p \diffd s \right)^{2/p}}{\theta^{2/p} t} \geq \left(\int_0^\theta |f(s)|^p \diffd s \right)^{2/p}\]
and hence
\begin{multline*}
\{f\in C_0[0,\theta] : m\beta \int_0^\theta |f(s)|^p \diffd s - \frac{1}{2}\int_0^\theta f'(s)^2 \diffd s \geq K\}\\
\subset \{f\in C_0[0,\theta] : m\beta \left(\int_0^\theta f'(s)^2 \diffd s\right)^{p/2} - \frac{1}{2}\int_0^\theta f'(s)^2 \diffd s \geq K\}\\
\subset \{f\in C_0[0,\theta] : \int_0^\theta f'(s)^2 \diffd s \leq K'\}
\end{multline*}
for some $K'$ since $p/2<1$. But by (\ref{jensen}),
\begin{multline*}
\{f\in C_0[0,\theta] : \int_0^\theta f'(s)^2 \diffd s \leq K'\}\\
\subset \{f\in C_0[0,\theta] : \forall s,t \in[0,\theta], |f(s)-f(t)|\leq \sqrt{(t-s)K'} \}
\end{multline*}
and the Arzel\`a-Ascoli theorem tells us that this latter set is totally bounded. Thus the set
\[\{f\in C_0[0,\theta] : m\beta \int_0^\theta |f(s)|^p \diffd s - \frac{1}{2}\int_0^\theta f'(s)^2 \diffd s \geq \sup_{f\in D} K(f,\theta)+\delta\}\]
is totally bounded, but by upper-semicontinuity it is closed, and hence compact. Since it is disjoint from \mbox{$\{f\in C_0[0,\theta] : \exists g\in D \hbox{ with } f(s)=g(s) \hsl\forall s\in[0,\theta]\}$}, which is closed, there is a positive distance between the two sets.
\end{proof}

Before continuing, we need the following lemma.

\begin{lem}
For any $x,y\in\Rb$ and $p\in[0,2)$,
\[|x+y|^p \leq |x|^p + |y|^p + 2|x|^{p/2}|y|^{p/2}.\]
\end{lem}

\noindent
This result is entirely elementary; see \cite{roberts:thesis} for a proof.


\begin{prop}\label{exp_prop}
If $f\not\in F_N$, then for $A = B(f,\varepsilon)$, we have
\[\limsup_{T\to\infty}\frac{1}{T^{2q-1}}\log\Eb\big[|N_T(\bar A,\theta)|\big] \leq \sup_{g\in \bar A} K(g,\theta) + R_N(\varepsilon)\]
and
\[\liminf_{T\to\infty}\frac{1}{T^{2q-1}}\log\Eb\big[|N_T(A,\theta)|\big] \geq \sup_{g\in A} K(g,\theta) - R_N(\varepsilon)\]
as $T\to\infty$, where
\[R_N(\varepsilon):=\left\{\begin{array}{ll} 0 & \hbox{if } p=0 \\ 2m\beta\left(\frac{N^2+1}{\sqrt{N}}+\varepsilon\right)^{p/2}(2\varepsilon)^{p/2} + (2\varepsilon)^p &\hbox{if } p>0;\end{array}\right.\]
in particular $R$ is a deterministic function of $\varepsilon$ such that for each  $N$, $R_N(\varepsilon)\to 0$ as $\varepsilon\to 0$
\end{prop}

\begin{proof}
From Schilder's theorem (Theorem 5.1 of \cite{varadhan:large_devs_apps}) we have
\[\limsup_{T\to\infty}\frac{1}{T^{2q-1}}\log\Pb(\xi_T \in N_T(\bar A,\theta)) \leq - \inf_{f\in \bar A\cap H_1} \frac{1}{2}\int_0^\theta f'(s)^2 \diffd s\]
and
\[\liminf_{T\to\infty}\frac{1}{T^{2q-1}}\log\Pb(\xi_T \in N_T(A,\theta)) \geq - \inf_{f\in A\cap H_1} \frac{1}{2}\int_0^\theta f'(s)^2 \diffd s\]
Thus, by the many-to-one lemma,
\begin{align*}
\limsup_{T\to\infty}\frac{1}{T^{2q-1}}\log\Eb\big[|N_T(\bar A,\theta)|\big] &\leq \limsup_{T\to\infty} \frac{1}{T^{2q-1}}\log \Eb\left[e^{m\beta\int_0^{\theta T}|\xi_s|^p \diffd s} \ind_{\{\xi_T\in N_T(\bar A,\theta)\}}\right]\\
&\leq \sup_{g\in \bar A}m\beta\int_0^\theta |g(s)|^p \diffd s - \inf_{g\in \bar A\cap H_1}\frac{1}{2}\int_0^\theta g'(s)^2 \diffd s
\end{align*}
and similarly
\[\liminf_{T\to\infty}\frac{1}{T^{2q-1}}\log\Eb\big[|N_T(A,\theta)|\big] \geq \inf_{g\in A}m\beta\int_0^\theta |g(s)|^p \diffd s - \inf_{g\in A\cap H_1}\frac{1}{2}\int_0^\theta g'(s)^2 \diffd s.\]
Note that since $f\not\in F_N$, $\sup_{s\in[0,\theta]} |f(s)| \leq (N^2+1)/\sqrt{N}$ (split $[0,\theta]$ into $N^2+1$ intervals of equal width; then $f$ changes by at most $1/\sqrt{N}$ on each interval). Now fix $\delta>0$ and choose $h\in A\cap H_1$ such that
\[\int_0^\theta h'(s)^2 \diffd s \leq \inf_{g\in A\cap H_1} \int_0^\theta g'(s)^2 \diffd s + \delta.\]
For any $g_1, g_2\in A$,
\begin{align*}
\int_0^\theta |g_1(s)|^p \diffd s &\leq \int_0^\theta (|g_2(s)|+2\varepsilon)^p \diffd s\\
&\leq \int_0^\theta |g_2(s)|^p \diffd s + 2\int_0^\theta |g_2(s)|^{p/2}(2\varepsilon)^{p/2}\diffd s + \int_0^\theta(2\varepsilon)^p \diffd s\\
&\leq \int_0^\theta |g_2(s)|^p \diffd s + 2\left(\frac{N^2+1}{\sqrt{N}}+\varepsilon\right)^{p/2}(2\varepsilon)^{p/2} + (2\varepsilon)^{p/2}
\end{align*}
and similarly
\[\int_0^\theta |g_1(s)|^p \diffd s \geq \int_0^\theta |g_2(s)|^p \diffd s - R_N(\varepsilon).\]
Thus
\begin{multline*}
\inf_{g\in A}m\beta\int_0^\theta |g(s)|^p \diffd s - \inf_{g\in A\cap H_1}\frac{1}{2}\int_0^\theta g'(s)^2 \diffd s \\
\geq \sup_{g\in A} m\beta\int_0^\theta |g(s)|^p \diffd s - \inf_{g\in A\cap H_1}\frac{1}{2}\int_0^\theta g'(s)^2 \diffd s - R_N(\varepsilon)\\
\geq \sup_{g\in A\cap H_1} \left\{m\beta\int_0^\theta |g(s)|^p \diffd s - \frac{1}{2}\int_0^\theta g'(s)^2 \diffd s\right\} - R_N(\varepsilon)
\end{multline*}
and
\begin{multline*}
\sup_{g\in \bar A}m\beta\int_0^\theta |g(s)|^p \diffd s - \inf_{g\in \bar A\cap H_1}\frac{1}{2}\int_0^\theta g'(s)^2 \diffd s \\
\leq m\beta\int_0^\theta |h(s)|^p \diffd s - \frac{1}{2}\int_0^\theta h'(s)^2 \diffd s + R_N(\varepsilon) + \delta\\
\leq \sup_{g\in \bar A\cap H_1}\left\{m\beta\int_0^\theta |g(s)|^p \diffd s - \frac{1}{2}\int_0^\theta g'(s)^2 \diffd s\right\} + R_N(\varepsilon) + \delta.
\end{multline*}
Since $\delta>0$ was arbitrary, this gives the desired result.
\end{proof}

\begin{proof}[Proof of Theorem \ref{expected_growth_paths_thm}]
We start with the lower bound. Given an open set $A$, fix $\delta>0$ and choose $g\in A$ such that
\[|K(g,t)-\sup_{f\in A} K(f,t)|<\delta/2.\]
Now choose $N\in\mathbb{N}$ such that $g\not\in F_N$ (this is possible since $\bigcap_N F_N = \emptyset$) and then $\eps>0$ such that $B(g,\eps)\subseteq A$ and $R_N(\eps)<\delta/2$. Then by Proposition \ref{exp_prop},
\begin{align*}
&\liminf_{T\to\infty} T^{1-2q}\log \E[|N_T(A,\theta)|]\\
&\geq \liminf_{T\to\infty} T^{1-2q}\log \E[|N_T(B(g,\eps),\theta)|]\\
&\geq K(g,\theta) - R_N(\eps)\\
&\geq \sup_{f\in A} K(f,\theta) - \delta.
\end{align*}
Since $\delta>0$ was arbitrary, our lower bound follows.

We now proceed with the upper bound. Take a closed set $D$ and $\theta\in[0,1]$, and again fix $\delta>0$. By Lemma \ref{extreme_lem} we may first choose $N\in\mathbb{N}$ such that
\[\E[N_T(F_N,\theta)] \leq \exp\left( T^{2q-1} (\sup_{f\in D} K(f,\theta) - 1) \right)\]
for all large $T$. Let $D_N = D\setminus F_N$. By Lemma \ref{semicont_lem} we may choose $\eps>0$ such that
\[\sup_{f\in D_N^\varepsilon} K(f,\theta) + R_N(\varepsilon) \leq \sup_{f\in D_N} K(f,\theta) + \delta.\]
Then by Lemma \ref{totally_bdd_lem} we may choose $n$ and $f_1,\ldots,f_n$ such that
\[D_N \subset \bigcup_{i=1}^n B(f_i,\varepsilon) \subset D_N^\varepsilon.\]
But now by Proposition \ref{exp_prop} we have
\begin{align*}
&\limsup_{T\to\infty} \frac{1}{T^{2q-1}}\log\E[|N_T(D,\theta)|]\\
&\leq \limsup_{T\to\infty}
\frac{1}{T^{2q-1}}\log\left\{\E[|N_T(F_N,\theta)|] + \sum_{i=1}^n \E[|N_T(B(f_i,\eps),\theta)|]\right\}\\
&\leq \sup_{f\in D^\varepsilon} K(f,\theta) + R_N(\varepsilon)\\
&\leq \sup_{f\in D} K(f,\theta) + \delta
\end{align*}
where the last inequality is due to our choice of $N$ and $\epsilon$. Since $\delta>0$ was arbitrary, the upper bound follows and the proof is complete.
\end{proof}

\begin{proof}[Proof of Corollary~\ref{cor_unconstrained}]
Observe that if we take $D = \{ f \in C[0,1] : f(1) \in [a,b]  \}$ and $A = \{ f \in C[0,1] : f(1) \in (a,b) \}$  ($D$ is closed and $A$ open) then necessarily $\sup_{f\in D} K(f,t) =\sup_{f\in A} K(f,t).$ To see this observe that for any $f$ such that, say, $f(1)=a$ one can find a sequence $f_n \to f$ in $A$ such that $K(f_n,1) \to K(f,1)$ (just modify $f$ near time 1 to finish ever so slightly above $a$).
\end{proof}

\section{Proof of Theorem \ref{growth_paths_thm}}\label{growth_paths_sec}

\subsection{The heuristic for the lower bound in Theorem \ref{growth_paths_thm}}\label{lower_bound_section}
We want to show that $N_T(f,\varepsilon,\theta)$ cannot be too small for large $T$. Recall that for $f\in C[0,1]$ and $\theta\in[0,1]$, we defined
\[K(f,\theta):=\left\{\begin{array}{ll}m\beta\int_0^\theta |f(s)|^p \diffd s - \frac{1}{2}\int_0^\theta f'(s)^2 \diffd s & \hbox{ if } f\in H_1\\ -\infty & \hbox{ otherwise.}\end{array}\right.\]
which is the growth rate in expectation. We are going to show that the
almost sure behaviour is described by a rate function which differs by the presence of a truncation at the extinction time $\theta_0$.

\vspace{3mm}

\emph{Step 1.} Consider a small rescaled time $\eta$. How many particles are in $N_T(f,\varepsilon,\eta)$? If $\eta$ is much smaller than $\varepsilon$, then (with high probability) no particle has had enough time to reach anywhere near the edge of the tube (approximately distance $\varepsilon T$ from the origin) before time $\eta T$. Thus, with high probability,
\[|N_T(f,\varepsilon,\eta)| = |N(\eta T)|.\]
We can then give a very simple (and inaccurate!) estimate to show that for some $\nu>0$, with high probability,
\[|N(\eta T)| \geq \nu T.\]

\emph{Step 2.} Given their positions at time $\eta T$, the particles in
$N_T(f,\varepsilon,\eta)$ act independently. Each particle $u$ in this set
thus draws an independent branching Brownian motion. Let
$N_T(u,f,\varepsilon,\theta)$ be the set of descendants of $u$ that are in
$N_T(f,\varepsilon,\theta)$. How big is this set? Since $\eta$ is very
small, $u$ is close to the origin at time $\eta T$. Thus we may hope, for
a given $\delta>0$, to find some $\gamma<1$ such that (for each $u$)
\[\Pb\left(|N_T(u,f,\varepsilon,\theta)| < \exp(K(f,\theta)T^{2q-1} - \delta T^{2q-1})\right) \leq \gamma.\]

\emph{Step 3.} If $N_T(f,\varepsilon,\theta)$ is to be small, then each of the sets $N_T(u,f,\varepsilon,\theta)$ for $u\in N_T(f,\varepsilon,\eta)$ must be small. Thus
\[\Pb\left(|N_T(f,\varepsilon,\theta)| < \exp(K(f,\theta)T^{2q-1}-\delta T^{2q-1})\right) \leq \gamma^{\nu T},\]
and we may apply Borel-Cantelli to deduce our result along lattice times (that is, times $T_j$, $j\geq0$ such that there exists $\tau>0$ with $T_j - T_{j-1} = \tau$ for all $j\geq1$).

\emph{Step 4.} We carry out a simple tube-reduction argument to move to continuous time.  The idea here is that if the result were true on lattice times but not in continuous time, the number of particles in $N_T(f,\varepsilon,\theta)$ must fall dramatically at infinitely many non-lattice times. We simply rule out this possibility using standard properties of Brownian motion.

\vspace{3mm}

The most difficult part of the proof is step 2. However, the spine results of Section \ref{spine_section} will simplify our task significantly.

\subsection{The proof of the lower bound in Theorem \ref{growth_paths_thm}}

We begin with step 1 of our heuristic, considering the size of $N_T(f,\varepsilon,\eta)$ for small $\eta$. First we will need the following simple lemma.

\begin{lem}\label{local_time_lem}
For any $t,\delta>0$ and $k>0$,
\[\Pt\left(\int_0^t \ind_{\{\xi_s \in (-\delta, \delta)\}} \diffd
s > k\right) \leq 3e^{t/2 - k/(4\delta)}.\]
\end{lem}

\begin{proof}
Defining $h_\delta: \Rb\to\Rb$ by
\[h_\delta(x) := \left\{\begin{array}{ll} |x| & \hbox{if } |x|\geq \delta \\ \frac{\delta}{2} + \frac{x^2}{2\delta} & \hbox{if } |x|<\delta \end{array}\right.\]
we have, by approximating with $C^2$ functions and applying It\^o's formula, that
\[h_\delta(\xi_t) = \frac{\delta}{2} + \int_0^t h'_\delta (\xi_s) \diffd\xi_s + \frac{1}{2\delta}\int_0^t \ind_{\{\xi_s\in(-\delta,\delta)\}} \diffd s\]
(this function $h_\delta$ is often seen when studying local times of Brownian motion --- for a full proof of the above see, for example, Lemma 4.11 of \cite{roberts:thesis}). Also
\[\Pt[e^{-\int_0^t h'_\delta(\xi_s)\diffd\xi_s}] \leq \Pt[e^{-\int_0^t
h'_\delta(\xi_s)\diffd\xi_s - \frac{1}{2}\int_0^t h'_\delta(\xi_s)^2 \diffd s}]e^{t/2} \leq e^{t/2}.\]
Thus
\begin{align*}
\Pt\left(\int_0^t \ind_{\{\xi_s \in (-\delta,\delta)\}} \diffd s > k\right)
&= \Pt\left(h_\delta(\xi_t) - \frac{\delta}{2} - \int_0^t h_\delta '(\xi_s)
\diffd\xi_s > \frac{k}{2\delta}\right)\\
&\leq \Pt\left(|\xi_t| - \int_0^t h_\delta ' (\xi_s) \diffd\xi_s > \frac{k}{2\delta}\right)\\
&\leq \Pt\left(|\xi_t| > \frac{k}{4\delta}\right) + \Pt\left(-\int_0^t
h_\delta ' (\xi_s) \diffd\xi_s > \frac{k}{4\delta}\right)\\
&\leq \Pt\left[e^{|\xi_t|}\right]e^{-k/(4\delta)} + \Pt\left[e^{- \int_0^t
h_\delta ' (\xi_s) \diffd\xi_s}\right]e^{-k/(4\delta)}\\
&\leq 3e^{t/2 - k/(4\delta)},
\end{align*}
establishing the result.
\end{proof}

\begin{lem}\label{rightmost_lem}
For any continuous $f$ with $f(0)=0$ and any $\varepsilon>0$, there exist $\eta>0$, $\nu>0$, $k>0$ and $T_1$ such that for all $T\geq T_1$,
\[\Pb(|N_T(f,\varepsilon/2,\eta)|<\nu T) \leq e^{-kT}.\]
\end{lem}

\begin{proof}
We first show that there exist $\eta>0$, $k_1>0$ and $T_1$ such that
\[\Pb(\exists u\in N(\eta T) : u\not\in N_T(f,\varepsilon/2,\eta)) \leq e^{-k_1 T} \hs \hs \forall T\geq T_1.\]
Choose $\eta$ small enough that $\sup_{s\in[0,\eta]} |f(s)|<\varepsilon/4$. Then, using the many-to-one lemma (at ($\star$)) and standard properties of Brownian motion,
\begin{align*}
&\Pb(\exists u\in N(\eta T) : u\not\in N_T(f,\varepsilon/2,\eta))\\
& = \Pb(\exists u\in N(\eta T),\hsl s\leq\eta : |X_u(sT) - T^q f(s)|\geq \varepsilon T^q/2)\\
&\leq \Pb(\exists u\in N(\eta T) : \sup_{s\leq\eta T}|X_u(s)| \geq \varepsilon T^q/4)\\
&\leq \sum_{k\geq1}\Pb(\exists u\in N(\eta T) : \sup_{s\leq\eta T} |X_u(s)| \in[k\varepsilon T^q/4,(k+1)\varepsilon T^q/4])\\
&\leq \sum_{k\geq1}e^{m\beta\int_0^{\eta T}((k+1)\varepsilon T^q)^p \diffd s}\Pb(\sup_{s\leq\eta T}|\xi_s|\in[k\varepsilon T^q/4,(k+1)\varepsilon T^q/4]) \tag{$\star$}\\
&\leq \sum_{k\geq1}4e^{m\beta(k+1)^p\varepsilon^p\eta T^{qp+1}}\Pb(\xi_{\eta T}\in[k\varepsilon T^q/4,(k+1)\varepsilon T^q/4])\\
&\leq \sum_{k\geq1}\frac{4}{\sqrt{2\pi\eta T}}\exp\left(m\beta(k+1)^p\varepsilon^p\eta T^{qp+1} - \frac{(k\varepsilon T^q)^2}{32\eta T}\right)\\
&\leq \sum_{k\geq1}\frac{4}{\sqrt{2\pi\eta
T}}\exp\left((m\beta\varepsilon^p\eta - \varepsilon^2/(32\eta))kT^{2q-1}\right)
\end{align*}
for sufficiently small $\eta$.
For small $\eta$ this is approximately
\[C \exp\left((m\beta\varepsilon^p\eta - \varepsilon^2/(32\eta))T^{2q-1}\right),\]
which gives a decay of $\exp(-k_1 T^{2q-1})$, which is more than
the decay required. We now aim to show that for any $\eta>0$, there exist $\nu>0$ and $k_2>0$ such that
\[\Pb(N(\eta T) < \nu T) \leq e^{-k_2 T}.\]
Indeed, if we let $n(t)$ be the number of births along the spine by time $t$, then certainly
\begin{align*}
&\Pb(N(\eta T) < \nu T)\\
&\leq \Pb(n(\eta T) < \nu T)\\
&\leq \Pb\left(\int_0^{\eta T} \ind_{\{\xi_s \in
[-(4\nu/(\beta\eta))^{1/p},(4\nu/(\beta\eta))^{1/p}]\}} \diffd s \geq \frac{1}{2}\eta T \right)\\
&\hspace{5mm}+ \Pb\left(\int_0^{\eta T} \ind_{\{\xi_s \in
[-(4\nu/(\beta\eta))^{1/p},(4\nu/(\beta\eta))^{1/p}]\}} \diffd s < \frac{1}{2}\eta T, \hsl n(\eta T) < \nu T \right).
\end{align*}
Lemma \ref{local_time_lem} shows that
\[\Pb\left(\int_0^{\eta T} \ind_{\{\xi_s \in
[-(4\nu/(\beta\eta))^{1/p},(4\nu/(\beta\eta))^{1/p}]\}} \diffd s \geq
\frac{1}{2}\eta T \right)\leq 3 \exp\left(\frac{\eta T}{2} - \frac{\eta
T}{8(4\nu/(\beta\eta))^{1/p})}\right)\]
so we have exponential decay in the first term provided that $\nu
< \beta\eta / 4^{p+1}$; and since births along the spine occur at rate at least $4\nu/\eta$ outside the interval
\[[-(4\nu/(\beta\eta))^{1/p},(4\nu/(\beta\eta))^{1/p}]\]
the second term is bounded above by the probability that a Poisson random variable with mean $2\nu T$ is less than $\nu T$. Let $Y\sim\hbox{Po}(2\nu T)$; then
\[\ind_{Y\leq\nu T}=\ind_{\exp(\nu T)\geq \exp(Y)} \leq \frac{e^{\nu T}}{e^{Y}}\]
so
\[P(Y\leq \nu T) \leq e^{\nu T}E[e^{-Y}] = e^{\nu T + 2\nu T(\exp(-1)-1)}\]
and this exponent is negative, so the second term also decays exponentially. Finally,
\[\Pb(|N_T(f,\varepsilon/2,\eta)|<\nu T) \leq \Pb(\exists u\in N(\eta T) : u\not\in N_T(f,\varepsilon/2,\eta)) + \Pb(N(\eta T) < \nu T)\]
and the proof is complete.
\end{proof}

We now move on to step 2, using the results of Section \ref{spine_section} to bound the probability that we have a small number of particles strictly below 1. The bound given is extremely crude, and there is much room for manoeuvre in the proof, but any improvement would only add unnecessary detail. The proof of this lemma runs exactly as for the corresponding result in the $p=0$ case seen in \cite{harris_roberts:scaled_growth}, with no extra technicalities; we include it again here for completeness.

\begin{lem}\label{bound_prob_below_1_lem}
If $f\in C^2[0,1]$, with $f(0)=0$, and $K(f,s)>0$ $\forall s\in(0,\theta]$, then for any $\varepsilon>0$ and $\delta>0$ there exists $T_0\geq0$ and $\gamma<1$ such that
\[\Pb\left(|N_T(f,\varepsilon,\theta)|<e^{K(f,\theta)T^{2q-1}-\delta T^{2q-1}}\right)\leq \gamma \hs\hs \forall T\geq T_0.\]
\end{lem}

\begin{proof}
Note that by Lemma \ref{int_by_parts_lem} for small enough $\varepsilon>0$ and large enough $T$,
\[|N_T(f,\varepsilon,\theta)|e^{-K(f,\theta)T^{2q-1}+\delta T^{2q-1}/2} \geq Z_T(\theta T)\]
and hence
\[\Pb\left(|N_T(f,\varepsilon,\theta)| < e^{K(f,\theta)T^{2q-1}-\delta T^{2q-1}}\right) \leq \Pb\left(Z_T(\theta T) < e^{-\delta T^{2q-1}/2}\right).\]
Suppose first that $f'(0)=0$. Then $\Eb[Z_T(\theta T)]=1$ and, again for small enough $\varepsilon$, by Proposition \ref{UI} the set $\{Z_T(t), T\geq 1, t\in[1,\theta T]\}$ is uniformly integrable. Thus we may choose $K$ such that
\[\sup_{T\geq1}\Eb[Z_T(\theta T) \ind_{\{Z_T(\theta T) > K\}}] \leq 1/4,\]
and then
\begin{align*}
1 = \Eb[Z_T(\theta T)] &= \Eb[Z_T(\theta T)\ind_{\{Z_T(\theta T)\leq 1/2\}}] + \Eb[Z_T(\theta T)\ind_{\{1/2< Z_T(\theta T)\leq K\}}]\\
&\hspace{20mm} + \Eb[Z_T(\theta T)\ind_{\{Z_T(\theta T) > K\}}]\\
&\leq 1/2 + K\Pb(Z_T(\theta T) > 1/2 ) + 1/4
\end{align*}
so that
\[\Pb(Z_T(\theta T) > 1/2) \geq 1/(4K).\]
Hence for large enough $T$,
\[\Pb\left(|N_T(f,\varepsilon,\theta)| < e^{K(f,\theta)T^{2q-1}-\delta T^{2q-1}}\right) \leq 1 - 1/(4K).\]
This is true for all small $\varepsilon>0$; but increasing $\varepsilon$ only increases $|N_T(f,\varepsilon,\theta)|$ so the statement holds for all $\varepsilon>0$. Finally, if $f'(0)\neq0$ then choose $g\in C^2[0,\theta]$ such that $g(0)=g'(0)=0$, $\sup_{s\leq\theta}|f-g|\leq\varepsilon/2$, $K(g,\phi)>0$ $\forall \phi\leq\theta$ and $K(g,\theta)> K(f,\theta)-\delta/2$ (for small $\eta$, the function
\[g(t):= \left\{\begin{array}{ll}f(t) + at + bt^2 + ct^3 + dt^4 & \hbox{if }t\in[0,\eta)\\ f(t) & \hbox{if } t\in[\eta,1]\end{array}\right.\]
will work for suitable $a,b,c,d\in\mathbb{R}$). Then
\begin{align*}
\Pb(|N_T(f,\varepsilon,\theta)| < e^{K(f,\theta)T^{2q-1}-\delta T^{2q-1}}) &\leq \Pb(|N_T(g,\varepsilon/2,\theta)| < e^{K(g,\theta)T^{2q-1}-\delta T^{2q-1}/2})\\
&\leq 1-1/(4K)
\end{align*}
as required.\end{proof}

We are now ready to carry out step 3 of the heuristic. Again this runs exactly as in \cite{harris_roberts:scaled_growth}.

\begin{prop}\label{lower_prop}
Suppose that $f\in C^2[0,1]$ and $K(f,s)>0$ $\forall s\in(0,\theta]$. Then for lattice times $T_j$
(recall that this means that there exists $\tau>0$ such that  $T_j-T_{j-1}= j\tau  $ for all $j\geq1$),
\[\liminf_{j\to\infty}\frac{1}{T_j^{2q-1}}\log |N_{T_j}(f,\varepsilon,\theta)|\geq K(f,\theta)\]
almost surely.
\end{prop}

\begin{proof}
For a particle $u$, recall that $u< v$ means that $u$ is an ancestor of
$v$ and  define
\[N_T(u,f,\varepsilon,\theta) := \{v\in N(\theta T) : u< v, \hs |X_v(tT)-T^q f(t)|<\varepsilon T^q \hs \forall t\in[0,\theta ]\},\]
the set of descendants of $u$ that are in $N_T(f,\varepsilon,\theta)$.
Then for $\delta >0$ and $\eta \in[0,\theta]$,
\begin{align*}
\Pb&\left(\left.|N_T(f,\varepsilon,\theta)| < e^{K(f,\theta)T^{2q-1}-\delta T^{2q-1}} \right| \Fg_{\eta T}\right)\\
&\leq \prod_{u\in N_T(f,\varepsilon/2,\eta)}\Pb\left(\left.|N_T(u,f,\varepsilon,\theta)|<e^{K(f,\theta)T^{2q-1} - \delta T^{2q-1}} \right| \Fg_{\eta T}\right)\\
&\leq \prod_{u\in N_T(f,\varepsilon/2,\eta)}\Pb\left(|N_T(g,\varepsilon/2,\theta-\eta)|<e^{K(f,\theta)T^{2q-1}-\delta T^{2q-1}}\right)
\end{align*}
since, given $\Fg_{\eta T}$, $\{|N_T(u,f,\varepsilon,\theta)| : u\in N_T(f,\varepsilon/2,\eta)\}$ are independent random variables, and where $g:[0,1]\to\Rb$ is any twice continuously differentiable extension of the function
\[\begin{array}{rrcl} g: & [0,\theta-\eta] &\to& \Rb\\
                          & t               &\to& f(t+\eta)-f(\eta).\end{array}\]
If $\eta$ is small enough, then
\[|K(f,\theta) - K(g,\theta - \eta)|<\delta/2\]
and
\[K(g,s) > 0 \hs\hs \forall s\in(0,\theta-\eta].\]
Hence, applying Lemma \ref{bound_prob_below_1_lem}, there exists $\gamma<1$ such that for all large $T$,
\begin{align*}
\Pb&\left(|N_T(g,\varepsilon/2,\theta-\eta)|< e^{K(f,\theta)T^{2q-1}-\delta T^{2q-1}}\right)\\
&\leq \Pb\left(|N_T(g,\varepsilon/2,\theta-\eta)| < e^{K(g,\theta-\eta)T^{2q-1}-\delta T^{2q-1}/2}\right)\\
&\leq \gamma.
\end{align*}
Thus for large $T$,
\begin{equation}\label{cond_prob}
\Pb\left(\left.|N_T(f,\varepsilon,\theta)| < e^{K(f,\theta)T^{2q-1}-\delta T^{2q-1}} \right| \Fg_{\eta T}\right) \leq \gamma^{|N_T(f,\varepsilon/2,\eta)|}.
\end{equation}
Taking expectations in (\ref{cond_prob}), and then applying Lemma \ref{rightmost_lem}, for small $\eta$ and some $\nu, k>0$, for large $T$ we have
\begin{align*}
\Pb&\left(|N_T(f,\varepsilon,\theta)| < e^{K(f,\theta)T^{2q-1}-\delta T^{2q-1}} \right)\\
&\leq \Pb\left(|N_T(f,\varepsilon/2,\eta)| < \nu T \right) + \gamma^{\nu T}\\
&\leq e^{-kT} + \gamma^{\nu T}.
\end{align*}
The Borel-Cantelli lemma now tells us that for any lattice times $T_j$, $j\geq0$,
\[\Pb\left(\liminf_{j\to\infty}\frac{1}{T_j^{2q-1}}\log|N_j(f,\varepsilon,\theta)| < K(f,\theta)-\delta \right) = 0,\]
and taking a union over $\delta>0$ gives the result.
\end{proof}

We now move to continuous time using step 4 of our heuristic.

\begin{prop}\label{tube_reduction}
Suppose that $f\in C^2[0,1]$ and $K(f,s)>0$ $\forall s\in(0,\theta]$. Then
\[\liminf_{T\to\infty}\frac{1}{T^{2q-1}}\log |N_T(f,\varepsilon,\theta)|\geq K(f,\theta)\]
almost surely.
\end{prop}

\begin{proof}
We claim first that for large enough $j\in\mathbb{N}$, provided that $T_1\leq1$,
\begin{multline*}
\left\{|N_{T_j}(f,\varepsilon,\theta)|>\inf_{T\in[T_j,T_{j+1}]}|N_T(f,2\varepsilon,\theta|\right\} \\
 \subset \left\{\exists v\in N_{T_j}(f,\varepsilon,\theta), u\in N(\theta T_{j+1}) : v< u, \sup_{T\in[T_j,T_{j+1}]}\hspace{-1.5mm} |X_u(\theta T)-X_u(\theta T_j)| > \frac{\varepsilon T_j^q }{ 2}\right\}.
\end{multline*}
Indeed, if $v\in N_{T_j}(f,\varepsilon,\theta)$, $T\in[T_j,T_{j+1}]$ and $S\in[0,\theta T]$ then for any descendant $u$ of $v$ at time $\theta T$,
\begin{align*}
|X_u(S) - T^q f(S/T)| &\leq |X_u(S) - X_u(S\wedge \theta T_j)| + |X_u(S\wedge \theta T_j)-T_j^q f(\textstyle{\frac{S\wedge \theta T_j}{T_j}})|\\
&\hspace{10mm} + |T_j^q f(\textstyle{\frac{S\wedge \theta T_j}{T_j}}) - T_j^q f(S/T)| + |T_j^q f(S/T) - T^q f(S/T)|\\
&\leq |X_u(S) - X_u(S\wedge \theta T_j)| + \varepsilon T_j^q\\
&\hspace{10mm} + T_j^q\sup_{\substack{x,y\in[0,\theta]\\|x-y|\leq 1/T_j}} |f(x)-f(y)| + \sup_{x\in[0,\theta]}|f(x)||T_{j+1}^q - T_j^q|\\
&\leq |X_u(S) - X_u(S\wedge \theta T_j)| + \frac{3\varepsilon}{2}T_j^q \hs\hs \hbox{ for large j;}
\end{align*}
so that if any particle is in $N_{T_j}(f,\varepsilon,\theta)$ but does not have a descendant in $N_T(f,2\varepsilon,\theta)$ then its descendants must satisfy
\[\sup_{S\in[\theta T_j, \theta T_{j+1}]} |X_u(S) - X_u(T_j)| \geq \varepsilon T_j^q/ 2.\]
This is enough to establish the claim, and we deduce via the many-to-one lemma plus Lemma \ref{xp_lem} and standard properties of Brownian motion (see Proposition 4.15 of \cite{roberts:thesis} for a more detailed justification) that
\begin{multline*}
\Pb\left(|N_{T_j}(f,\varepsilon,\theta)|>\inf_{T\in[T_j,T_{j+1}]}|N_T(f,2\varepsilon,\theta)|\right)\\
\leq 4e^{m\beta T_{j}^{2q-1} \sup_{g\in B(f,\varepsilon)}\int_0^\theta
|g(s)|^p \diffd s}
\sum_{k=1}^\infty e^{-(k\varepsilon T_j^q)^2/(8\theta T_1) + m\beta T_j^{2q-2}(|f(\theta)|+(k+3)\varepsilon/2)}
\end{multline*}
which, as in Lemma \ref{rightmost_lem}, is exponentially small in $T_j$. Thus the probabilities are summable and we may apply Borel-Cantelli to see that
\[\Pb(|N_{T_j}(f,\varepsilon,\theta)| > \inf_{T\in[T_j,T_{j+1}]}|N_T(f,2\varepsilon,\theta)| \hbox{ infinitely often} )=0.\]
Now,
\begin{multline*}
\Pb\left(\liminf_{T\to\infty}\frac{1}{T^{2q-1}}\log |N_T(f,\varepsilon,\theta)| < K(f,\theta) \right)\\
\leq \Pb\left(\liminf_{j\to\infty}\frac{1}{T_j^{2q-1}}\log |N_{T_j}(f,2\varepsilon,\theta)| < K(f,\theta)\right)\\
+ \Pb\left(\liminf_{j\to\infty} \frac{\inf_{T\in[T_j,T_{j+1}]}|N_T(f,\varepsilon,\theta)|}{|N_{T_j}(f,2\varepsilon,\theta)|} < 1 \right)
\end{multline*}
which is zero by Proposition \ref{lower_prop} and the above.
\end{proof}

This gives us our desired lower bound for Theorem \ref{growth_paths_thm}.

\begin{cor}\label{lower_bd}
For any open set $A\subset C[0,1]$ and $\theta\in[0,1]$, we have
\[\liminf_{T\to\infty}\frac{1}{T^{2q-1}}\log|N_T(A,\theta)| \geq \sup
\{  K(f,\theta) : f\in A, \theta_0(f) \ge \theta \}\]
almost surely.
\end{cor}

\begin{proof}
Clearly if $\sup \{  K(f,\theta) : f\in A, \theta_0(f) \ge \theta \}=-\infty$ then there is nothing to prove. Thus it suffices to consider the case when there exists $f\in A$ such that $f\in H_1$ and $\theta\leq\theta_0(f)$. Since $A$ is open, in this case we can in fact find $f\in A$ such that $K(f,s)>0$ $\forall s\in(0,\theta]$ (if $K(f,\phi)=0$ for some $\phi\leq\theta$, just choose $\eta$ small enough that $(1-\eta)f \in A$) and such that $f$ is twice continuously differentiable on $[0,1]$ (the twice continuously differentiable functions are dense in $C[0,1]$). Thus necessarily $\sup\{  K(f,\theta) : f\in A, \theta_0(f) \ge \theta \} >0$, and for any $\delta>0$ we may further assume that $K(f,\theta)>\sup\{  K(f,\theta) : f\in A, \theta_0(f) \ge \theta \}-\delta$. Again since $A$ is open, we may take $\varepsilon$ such that $B(f,\varepsilon)\subset A$;
then clearly for any $T$
\[N_T(f,\varepsilon,\theta)\subset N_T(A,\theta)\]
so by Proposition \ref{tube_reduction} we have
\[\liminf_{T\to\infty}\frac{1}{T^{2q-1}}\log N_T(A,\theta) \geq \sup\{  K(f,\theta) : f\in A, \theta_0(f) \ge \theta \}-\delta\]
almost surely, and by taking a union over $\delta>0$ we may deduce the result.
\end{proof}

\subsection{The upper bound in Theorem \ref{growth_paths_thm}}\label{upper_bound_section}
Our plan is as follows: we recall that we ruled out the possibility of any particles following unusual paths in Lemma \ref{extreme_lem}, which allows us to restrict our attention to a compact set, and hence small balls about sensible paths. We then carry out the task of obtaining a bound along lattice times for balls about such paths in Proposition \ref{upper_prop}. By expanding these balls slightly (using an argument similar to that in Proposition \ref{tube_reduction}) we may then bound the growth in continuous time; this is done in Lemma \ref{upper_lattice_to_cts}, and finally we draw this work together in Proposition \ref{upper_bd_J} to give the bound in continuous time for any closed set $D$.

For simplicity of notation, we break with convention by letting
\[\|f\|_\theta:=\sup_{s\in[0,\theta]}|f(s)|\]
for $f\in C[0,\theta]$ or $f\in C[0,1]$ (on this latter space, $\|\cdot\|_\theta$ is clearly not a norm, but this will not matter to us). We also extend the definition of $N_T(D,\theta)$ to sets $D\subset C[0,\theta]$ in the obvious way, setting
\[N_T(D,\theta) := \{u \in N(\theta T) : \exists f \in D \hbox{ with } X_u(tT) = T^q f(t) \hs \forall t\in[0,\theta ]\}.\]
With a slight abuse of notation, for $D\subset C[0,1]$ and $\theta\in[0,1]$ we define
\[K(D,\theta) := \sup_{f\in D} K(f,\theta).\]

We now attempt to establish an upper bound along lattice times for closed balls about functions outside $F_N$.
Recall the definition of $F_N$ from Lemma \ref{extreme_lem} and that of $R_N(\varepsilon)$ from Proposition \ref{exp_prop}.

\begin{prop}\label{upper_prop}
Fix $N\in \mathbb{N}$. For any closed ball $D=\overline{B(f,\varepsilon)} \subset C[0,1]$ about any $f\not\in F_N$, and any $\theta\in[0,1]$ and lattice times $T_j$, we have
\[\limsup_{j\to\infty}\frac{1}{T_j^{2q-1}}\log|N_{T_j}(D,\theta)| \leq K(D,\theta) + R_N(\varepsilon)\]
almost surely.
\end{prop}

\begin{proof}
Proposition \ref{exp_prop} tells us that
\[\limsup_{T\to\infty}\frac{1}{T^{2q-1}}\log\Eb\big[|N_T(D,\theta)|\big]\leq K(D,\theta) + R_N(\varepsilon).\]
Applying Markov's inequality, for any $\delta>0$ and $p\in[0,2)$ we get
\begin{multline*}
\limsup_{T\to\infty}\frac{1}{T^{2q-1}}\log \Pb\big(|N_T(D,\theta)| \geq e^{K(D,\theta)T^{2q-1} + R_N(\varepsilon)T^{2q-1} + \delta T^{2q-1}}\big)\\
\leq \limsup_{T\to\infty}\frac{1}{T^{2q-1}}\log \frac{\Eb\big[|N_T(D,\theta)|\big]}{e^{K(D,\theta)T^{2q-1} + R_N(\varepsilon)T^{2q-1} + \delta T^{2q-1}}} \leq -\delta
\end{multline*}
so that for lattice times $T_1, T_2, \ldots$ we have
\[\sum_{j=1}^\infty \Pb\big(|N_{T_j}(D,\theta)| \geq e^{K(D,\theta)T_j^{2q-1} + R_N(\varepsilon)T_j^{2q-1} + \delta T_j^{2q-1}}\big) < \infty\]
and hence by the Borel-Cantelli lemma
\[\Pb\left(\limsup_{j\to\infty}\frac{1}{T_j^{2q-1}}\log|N_{T_j}(D,\theta)| \geq K(D,\theta)+R_N(\varepsilon)+\delta\right)=0.\]
Taking a union over $\delta>0$ now gives the result.
\end{proof}

We now check that an upper bound holds in continuous time. For $\varepsilon>0$ and $D\subset C[0,1]$, define
\[D^\varepsilon := \{f\in C[0,1] : \exists g\in D \hbox{ with
} \|f-g\|\leq\varepsilon\}.\]

\begin{lem}\label{upper_lattice_to_cts}
If $D=\overline{B(f,\varepsilon)}\subset C[0,1]$ for some $f\not\in F_N$, then
\[\limsup_{T\to\infty}\frac{1}{T^{2q-1}}\log|N_T(D,\theta)|\leq K(D^{\varepsilon},\theta) + R_N(2\varepsilon)\]
almost surely.
\end{lem}

\begin{proof}
First note that for lattice times $T_1,T_2,\ldots$,
\begin{multline*}
\Pb\left(\limsup_{T\to\infty}\frac{1}{T^{2q-1}}\log|N_T(D,\theta)|>K(D^{\varepsilon},\theta)+R_N(2\varepsilon)+\delta\right)\\
\leq \Pb\left(\limsup_{j\to\infty}\frac{1}{T_j^{2q-1}}\log|N_{T_j}(D^\varepsilon,\theta)| > K(D^\varepsilon,\theta)+R_N(2\varepsilon)\right)\\
+ \Pb\left(\limsup_{j\to\infty}\frac{1}{T_j^{2q-1}}\log\sup_{T\in[T_j,T_{j+1}]}\frac{|N_T(D,\theta)|}{|N_{T_j}(D^\varepsilon,\theta)|} > \delta\right).
\end{multline*}
Clearly $D^\varepsilon = \overline{B(f,2\varepsilon)}$, so immediately by Proposition \ref{upper_prop},
\[\Pb\left(\limsup_{j\to\infty}\frac{1}{T_j^{2q-1}}\log|N_{T_j}(D^\varepsilon,\theta)| > K(D^\varepsilon,\theta)+R_N(2\varepsilon)\right) = 0\]
and we may concentrate on the last term. We claim that for $j$ large enough, provided that $T_1\leq1$, for any $T\in[T_j,T_{j+1}]$ we have
\[u\in N_T(D,\theta) \Rightarrow \exists v< u \hbox{ with } v\in N_{T_j}(D^\varepsilon,\theta).\]
Indeed, if $u\in N_T(D,\theta)$ then for any $S\leq \theta T_j$,
\begin{align*}
|X_u(S) - T_j^q f(S/T_j)| &\leq |X_u(S) - T^q f\left(S/T\right)| + |T_j^q f\left(S/T_j\right) - T^q f\left(S/T_j\right)|\\
&\hspace{50mm} + T^q |f\left(S/T_j\right) - f\left(S/T\right)|\\
&\leq T^q \varepsilon + \|f\|_\theta (T_{j+1}^q - T_j^q) + T^q\sup_{\begin{subarray}{c}x,y\in[0,\theta]\\|x-y|\leq 1/T_j \end{subarray}}|f(x)-f(y)|
\end{align*}
which is smaller than $2\varepsilon T_j^q$ for large $j$ since $f$ is absolutely continuous.

We deduce that for large $j$ every particle in $N_T(D,\theta)$ for any $T\in[T_j,T_{j+1}]$ has an ancestor in $N_{T_j}(D^\varepsilon,\theta)$. We now use this fact to ensure that $N_T(D,\theta)$ cannot increase dramatically between times $T_j$ and $T_{j+1}$.

We temporarily need some more notation. For $T>S\geq 0$ and $u\in N(S)$, let $N(u,S,T)$ be the set of descendants of $u$ born between times $S$ and $T$. Also let $\Pt_x$ be the translation of $\Pt$ under which we start with one particle at $x$ rather than at the origin.
Then, using the Markov property and the many-to-one lemma,
for $j$ large enough,
\begin{align*}
&\Eb\left[\left.\sup_{T\in[T_j,T_{j+1}]}|N_T(D,\theta)| \right|\Fg_{\theta T_j}\right]\\
&\leq \Eb\left[\left. \sum_{u\in N_{T_j}(D^\varepsilon,\theta)} |N(u,\theta T_j,\theta T_{j+1})| \right| \Fg_{\theta T_j}\right]\\
&\leq \sum_{u\in N_{T_j}(D^\varepsilon,\theta)} \Eb_{X_u(\theta T_j)}\left[ |N(\theta T_1)| \right]\\
&= \sum_{u\in N_{T_j}(D^\varepsilon,\theta)} \Eb_{X_u(\theta T_j)}\left[ e^{m\beta \int_0^{\theta T_1} |\xi_s|^p \diffd s} \right]\\
&\leq \sum_{u\in N_{T_j}(D^\varepsilon, \theta)} \sum_{k\geq0} e^{m\beta\theta T_1(|X_u(\theta T_j)|+k+1)^p} \Pb_{X_u(\theta T_j)} \left(\sup_{S\in[0,\theta T_1]} |\xi_S - \xi_0| \in[k,k+1]\right)\\
&\leq |N_{T_j}(D^\varepsilon,\theta)| \sum_{k\geq0} e^{m\beta\theta
T_1(T_j^q(\|f\|_\theta+2\varepsilon)+k+1)^p} \frac{4e^{-k^2 / (2\theta
T_1)}}{\sqrt{2\pi\theta T_1}};
\end{align*}
since $p<2$ this sum converges, giving
\[\Eb\left[\left.\sup_{T\in[T_j,T_{j+1}]}|N_T(D,\theta)| \right|\Fg_{\theta T_j}\right] \leq |N_{T_j}(D^\varepsilon,\theta)|e^{O(T_j^{pq})}\]
where the $O(T_j^{pq})$ is deterministic.
But $pq=2q-2$ and by Markov's inequality
\begin{align*}
&\Pb\left(\sup_{T\in[T_j,T_{j+1}]}\frac{|N_T(D,\theta)|}{|N_{T_j}(D^\varepsilon,\theta)|} > \exp\left(\delta T_j^{2q-1}\right) \right)\\
&\leq \Eb\left[\frac{\Eb\left[\left.\sup_{T\in[T_j,T_{j+1}]}|N_T(D,\theta)| \right|\Fg_{\theta T_j}\right]}{|N_{T_j}(D^\varepsilon,\theta)|}\right]\exp(-\delta T_j^{2q-1})\\
&\leq \exp(O(T_j^{2q-2}) - \delta T_j^{2q-1}).
\end{align*}
Thus we may apply Borel-Cantelli to see that
\[\Pb\left(\limsup_{j\to\infty}\frac{1}{T_j^{2q-1}}\log\sup_{T\in[T_j,T_{j+1}]}\frac{|N_T(D,\theta)|}{|N_{T_j}(D^\varepsilon,\theta)|} > \delta\right)=0.\]
Again taking a union over $\delta>0$ gives the result.
\end{proof}

We are now in a position to give an upper bound for any closed set $D$ in continuous time. This upper bound is not quite what we asked for in Theorem \ref{growth_paths_thm}, but the final step --- truncating $K$ at $\theta_0$ --- will be carried out in Corollary \ref{upper_bd}.

\begin{prop}\label{upper_bd_J}
If $D\subset C[0,1]$ is closed, then for any $\theta\in[0,1]$
\[\limsup_{T\to\infty}\frac{1}{T^{2q-1}}\log|N_T(D,\theta)|\leq K(D,\theta)\]
almost surely.
\end{prop}

\begin{proof}
Clearly (since our first particle starts from 0) $N_T(D\setminus C_0[0,1],\theta)=\emptyset$ for all $T$, so we may assume without loss of generality that $D\subset C_0[0,1]$. Now fix $\delta>0$ and choose $N$ (by Lemma \ref{extreme_lem}) such that
\[\limsup_{T\to\infty}\frac{1}{T^{2q-1}}\log |N_T(F_N,\theta)|= -\infty, \text{ a.s.}\]
By Lemma \ref{semicont_lem} and the fact that $R_N(2\varepsilon)\to 0$ as
$\varepsilon\to 0$, we may choose $\varepsilon>0$ such that
$K(D^\varepsilon,\theta) + R_N(2\varepsilon) < K(D,\theta)+\delta$. Then,
by Lemma \ref{totally_bdd_lem}, for any $N$ and some $n$ (depending on $N$)
and $f_k \in C[0,1]\setminus F_N$, $k=1,2,\ldots,n$,
\begin{align*}
&\Pb\left(\limsup_{T\to\infty}\frac{1}{T^{2q-1}}\log|N_T(D,\theta)|>K(D,\theta)+\delta\right)\\
&\leq \Pb\left(\limsup_{T\to\infty}\frac{1}{T^{2q-1}}\log|N_T(F_N,\theta)| > K(D,\theta)+\delta\right)\\
&\hspace{15mm} + \sum_{k=1}^n \Pb\left(\limsup_{T\to\infty}\frac{1}{T^{2q-1}}\log|N_T(f_k,\varepsilon,\theta)| > K(D^\varepsilon,\theta) + R_N(2\varepsilon)\right).
\end{align*}
By our choice of $N$, the first term on the right-hand side is zero, and by Lemma \ref{upper_lattice_to_cts} all of the terms in the sum are also zero. As usual we take a union over $\delta>0$ to complete the proof.
\end{proof}

\begin{cor}\label{upper_bd}
For any closed set $D\subset C[0,1]$ and $\theta\in[0,1]$, we have
\[\limsup_{T\to\infty}\frac{1}{T^{2q-1}}\log|N_T(D,\theta)| \leq
\sup\{K(f,\theta) : f\in D, \theta_0(f)\ge \theta \}\]
almost surely.
\end{cor}

\begin{proof}
Since $|N_T(D,\theta)|$ is integer valued,
\[\frac{1}{T^{2q-1}}\log |N_T(D,\theta)| < 0 \hs \Rightarrow \hs \frac{1}{T^{2q-1}}\log |N_T(D,\theta)| = -\infty.\]
Thus, by Proposition \ref{upper_bd_J}, if $K(D,\theta)<0$ then
\[\Pb\left(\limsup_{T\to\infty}\frac{1}{T^{2q-1}}\log|N_T(D,\theta)| > -\infty\right)=0.\]
Further, clearly for $\phi\leq\theta$ and any $T\geq0$, if $N_T(D,\phi)=\emptyset$ then necessarily $N_T(D,\theta)=\emptyset$. Thus if there exists $\phi\leq\theta$ with $K(D,\phi)<0$, then
\[\Pb\left(\limsup_{T\to\infty}\frac{1}{T^{2q-1}}\log|N_T(D,\theta)| > -\infty\right)=0\]
which completes the proof.
\end{proof}

\begin{proof}[Proof of Theorem \ref{growth_paths_thm}]
Combining Corollary \ref{lower_bd} with Corollary \ref{upper_bd} completes the proof.
\end{proof}

\begin{proof}[Proof of Corollary \ref{C:corol nas}]
Corollary \ref{C:corol nas} is proven in the same way as Corollary \ref{cor_unconstrained} (it is easily checked that the limsup and liminf of Theorem \ref{unconstrained_thm} agree on sets of the form $D_{z,\epsilon} := \{ f\in C[0,1] , |f(1)-z| \le \varepsilon \}$ and $A_{z,\epsilon} := \{ f\in C[0,1] , |f(1)-z| < \varepsilon \}$ for $\varepsilon >0$).
\end{proof}

\section{Proof of Theorem \ref{large_devs_thm}}\label{large_devs_sec}
There is much overlap in the proofs of Theorems \ref{large_devs_thm} and \ref{growth_paths_thm}, and so we will leave out many of the details and refer to earlier sections. Our task is also made easier by the fact that we are trying to prove a statement about one-dimensional distributions rather than the full sample paths of the BBM.

\begin{proof}[Proof of upper bound in Theorem \ref{large_devs_thm}]
Lemma \ref{extreme_lem} tells us that for any $k>0$, we may choose $N\in\mathbb{N}$ such that
\begin{equation}
\P(N_T(F_N,\theta)\neq\emptyset) \leq \exp\left(-kT^{2q-1}\right)
\label{P<ekT}
\end{equation}
(recall that for large $N$, $F_N$ was a set of extreme paths that were very difficult for particles to follow). Now, in Proposition \ref{exp_prop} we saw that for any $f\in D\setminus F_N$ and any $\phi\in[0,1]$,
\[\limsup_{T\to\infty}\frac{1}{T^{2q-1}}\log\Eb\big[|N_T(f,\varepsilon,\phi)|\big]\leq K(D^\varepsilon,\phi) + R_N(\varepsilon)\]
where for each $N$, $R_N(\varepsilon)\to0$ as $\varepsilon\to 0$.

Fix $\delta>0$, choose $k>0$ such that $-k < \inf_{\phi<\theta}K(D,\phi)$ and
$N$ large enough that \eqref{P<ekT} is verified.
By a similar argument to that in Lemma \ref{semicont_lem} (the upper-semicontinuity carries over easily to the function $f\mapsto\inf_{\phi\leq\theta} K(f,\phi)$) we may then choose $\varepsilon$ small enough that
\[\inf_{\phi\leq\theta}K(D^\varepsilon,\phi) + R_N(\varepsilon) \leq \inf_{\phi\leq\theta} K(D,\phi) + \delta.\]
Using compactness (see Lemma \ref{totally_bdd_lem}), we can choose
$n\in\mathbb{N}$ and $f_1,f_2,\ldots,f_n \in D\setminus F_N$ such that
\[D\subset F_N\cup B(f_1,\varepsilon)\cup\ldots\cup B(f_n,\varepsilon)\subset D^\varepsilon.\]
Now for $\phi\leq\theta$ and any set $B$,
\[\{N_T(B,\phi)=\emptyset\} \subseteq \{ N_T(B,\theta)=\emptyset\}\]
so, for $T$ large enough (depending on the $f_1$, \ldots, $f_n$)
\begin{align*}
\P(N_T(D,\theta)\neq\emptyset)&\leq \P(N_T(F_N,\theta)\neq\emptyset)
+ \sum_{i=1}^n \P(N_T(f_i,\varepsilon,\theta)\neq\emptyset)\\
&\leq \P(N_T(F_N,\theta)\neq\emptyset)
+ \sum_{i=1}^n\inf_{\phi\leq\theta}\P(N_T(f_i,\varepsilon,\phi)\neq\emptyset)\\
&\leq e^{-kT^{2q-1}} + \sum_{i=1}^n\inf_{\phi\leq\theta}\E[|N_T(f_i,\varepsilon,\phi)|]\\
&\leq e^{-kT^{2q-1}} + \sum_{i=1}^n\inf_{\phi\leq\theta}e^{K(D^\varepsilon,\phi)T^{2q-1} + R_N(\varepsilon)T^{2q-1}}\\
&\leq (n+1)e^{\delta T^{2q-1}} \inf_{\phi\leq\theta}e^{K(D,\phi) T^{2q-1}}.
\end{align*}
Taking logarithms, dividing by $T^{2q-1}$ and letting $\delta\downarrow 0$ gives the desired result.
\end{proof}

\begin{proof}[Proof of lower bound in Theorem \ref{large_devs_thm}]
First note that it suffices to consider the case $A = B(f,\varepsilon)$ for $f\in C^2[0,1]$. Now
\begin{align*}
\P(N_T(A,\theta)\neq\emptyset) &= \P(Z_T(\theta T) > 0)\\
&=\Qt_T\left[\frac{1}{Z_T(\theta T)}\right]\\
&\geq \Qt_T\left[\frac{1}{\Qt_T[Z_T(\theta T)|\Gt_T]}\right].
\end{align*}
As in the proof of Lemma \ref{spine_domination} we can use Lemmas \ref{scaled_int_lem} and \ref{xp_lem} to bound the spine decomposition; for any $\delta>0$ we may choose $\varepsilon$ small enough that
\begin{align*}
\Qt_T[Z_T(\theta T) | \Gt_T] &\leq \sum_{u < \xi_{\theta T}} A_u e^{-\delta \int_0^{S_u} |\xi_s|^p
\diffd s - K(f,S_u/T)T^{2q-1} + 2\delta T^{2q-1}}\\
&\leq \sum_{u < \xi_{\theta T}} A_u e^{-\delta \int_0^{S_u} |\xi_s|^p
\diffd s - \inf_{\phi<\theta}K(f,\phi)T^{2q-1} + 2\delta T^{2q-1}}.
\end{align*}
Thus (since $\delta > 0$ was arbitrary) it certainly suffices to show that
\[\liminf_{T\to\infty} \Qt_T\left[\frac{1}{\sum_{u<\xi_{\theta T}} A_u e^{-\delta \int_0^{S_u}|\xi_s|^p \diffd s}}\right] > 0.\]
But using the auxiliary random variables $A_1, A_2, \ldots$ and $S_1, S_2, \ldots$ from the proof of Proposition \ref{UI} we have
\[
\Qt_T\left(
\limsup_{T\to\infty} \sum_{u<\xi_{\theta T}} A_u e^{-\delta \int_0^{S_u}|\xi_s|^p \diffd s}
<\infty\right)
 \geq
Q\left(
\sum_{n=1}^\infty A_n e^{-\delta \mathcal{S}_n} < \infty
\right) = 1.
\]
We are now done by Fatou's lemma.
\end{proof}

\section{Optimal paths: proofs of Theorems \ref{unconstrained_thm} and \ref{constrained_thm}}\label{sec:optpaths}

We start with Theorem~\ref{constrained_thm}.
We recall our optimisation problem.
\[
K(f,t) := \begin{cases}
\displaystyle\int_0^t \Big[ m\beta |f(s)|^p - \frac{1}{2} f'(s)^2 \Big]\diffd s   & \text{ if $f\in H_1$},
\\ - \infty & \text{otherwise}.\end{cases}
\]
where the class of functions in which we would like to optimise is
$$
H_1:=\{f:[0,1]\to\R : f(t)=\int_0^t h(s)\,\diffd s, h\in L^2[0,1]\}.
$$

\noindent\textbf{Optimisation Problem.} For $z\in\R$, find
$$
\sup_{f\in H_1, f(1)=z} K(f,1),
$$
subject to
\begin{equation}\label{eqn:constraint}
K(f,\theta)\geq 0 \quad \forall \theta\in[0,1].
\end{equation}

By symmetry it is sufficient to consider $z\ge 0$. Furthermore, if $g\in
H_1$ is any path satisfying the
constraint~\eqref{eqn:constraint}, then $|g|\in H_1$ also satisfies
the constraint, has the same value
at the end point, and $K(g,1)=K(|g|,1)$. So, without loss
of generality, we can assume for existence that
$g\geq 0$ and drop the modulus sign from the definition of $K$.

\medskip

\noindent\textbf{Optimal paths.}
For each $s \in [0,1]$ define a path $g \in H_1$ as follows
\begin{enumerate}
\item For all $t\in[0,s)$, we set $g(t)=r(t)$, the scaled position of the right-most particle (see (\ref{r(s)})). Otherwise said, $g$ is the solution to
$$
\frac12g'(t)^2=m\beta g(t)^p,\quad g(0)=0, \qquad t \le s.
$$

\item For $t\in[{s},1]$, $g$ satisfies
$$
g''(t)+m\beta pg^{p-1}(t)=0,
$$
with initial condition specified by the fact that $g$ and $g'$ are continuous at ${s}.$
\end{enumerate}

\medskip

We aim to show that, for each $z \in[0,\bar z]$
the above optimisation problem has a unique solution given
by the unique $g_z$ constructed as above by picking the unique $s= s_z\in[0,1]$ such that $g_z(1)=z.$ 


\medskip

The strategy for our proof is that we will first show that $g_z$ defined
above is the unique solution to the
optimisation problem in the smaller class $C^2_{\text{piecewise}}\subset H_1$ of functions which are piecewise $C^2$ on $[0,1]$ (note that the optimal functions $g_z$ have discontinuous second
derivative at $s_z$).
We will do this by exhibiting
a series of properties which a solution to the optimisation problem,
\emph{if one exists}, must
satisfy. The unique function satisfying all of these properties will
be our $g_z$. We will then
show that $g_z$ is also the optimum in the full class of functions $H_1$
by showing that if there
existed a better function in the larger class, there would also be a
better function in our restricted
class, and so obtain a contradiction.

For notational convenience, in places where it will not cause confusion, we will write $g$ for $g_z$.


\begin{lem}\label{lem:derivative_cts}
Any optimal (in $C^2_{\text{piecewise}}$) path $g$ is in fact $C^1.$
\end{lem}

\begin{proof}
As $g$ is $C^2_{\text{piecewise}}$, there exist $0=x_0 < x_1 <\ldots < x_K=1$ such that $g$ is $C^2$ on any interval $(x_i, x_{i+1}), i=0,\ldots ,K-1.$ It is therefore enough to show that $\forall i= 1 ,\ldots, K-1, g'(x_i-) = g'(x_i+).$ To simplify notations we just write $x$ for $x_i$ in the following.

Suppose that
$g'(x+)<g'(x-)$. We will show that it is possible to construct a better function. Choose $\eps>0$, and consider the
function $\tilde g$ defined by taking $g$, and interpolating linearly
on the interval $[x-\eps,x+\eps]$:
$$
\tilde g(t):=
\begin{cases}
g(t),\quad\quad& t\in[0,x-\eps)\\
g(x-\eps) +\frac{t-(x-\eps)}{2\eps}\Big(g(x+\eps)-g(x-\eps)\Big), & t\in[x-\eps,x+\eps)\\
g(t), &t\in[x+\eps,1].
\end{cases}
$$
For $t\in(x-\eps,x)$, we have $g'(t)=g'(x-)+O(\eps)$; similarly, for
$t\in(x,x+\eps)$, we have $g'(t)=g'(x+)+O(\eps)$. To simplify
notation, we define $g'_-:=g'(x-)$ and $g'_+:=g'(x+)$. Using Taylor
expansions again, we have
\begin{align*}
K(\tilde
g,1)-K(g,1)&=\int_{x-\eps}^{x+\eps}\left(\frac12g'(t)^2-\frac12\tilde
g'(t)^2\right)\diffd t+O(\eps^2)\\
&=\eps\Big(\frac12{g'_-}^2+\frac12{g'_+}^2-\frac14(g'_-+g'_+)^2 \Big)+O(\eps^2)\\
&=\frac\eps4(g'_--g'_+)^2+O(\eps^2)>0
\end{align*}
for all sufficiently small $\eps>0$. A very similar calculation shows that
$\tilde g$ satisfies~\eqref{eqn:constraint} for all $t\in(x-\eps,x+\eps)$, which
proves the result. The case $g'_+<g'_-$ is settled in the same manner.
\end{proof}

\begin{lem}\label{lem:optdiffeqn}
If $g$ is an optimal trajectory then, on
any interval $I$ of $[0,1]$  such that $K(g,t)>0$ for all $t\in I$, one has
\begin{equation}\label{eqn:optdiffeqn}
m\beta p g^{p-1}+g''=0.
\end{equation}
\end{lem}

\begin{proof}
Starting from an optimal trajectory $g$, consider the deformed trajectory
$g+\eps h$ where $\eps$ is small and $h$ is a sufficiently smooth
function with $h(0)=h(1)=0$ so that in particular the end-point is still
fixed at $z$. Using a Taylor expansion and an integration by parts we have
\begin{equation}
\begin{aligned}
K(g+\eps h,t)
	& = K(g,t) + \int_0^t \big[m\beta (g+\eps h)^p -m\beta
g^p-g'h'\eps\big]\diffd s +o(\eps),\\
	& = K(g,t) + \eps \int_0^t \big[m\beta p g^{p-1}+g''\big] h\,\diffd s
		-\eps g'(t) h(t) +o(\eps).
\end{aligned}
\label{deformed}
\end{equation}
Assume that $K(g,t)>0$ on some interval~$I$. Then for any $t_1<t_2$ in $I$, there exists $c>0$ such that $K(g,t)>c$ for all $t\in[t_1,t_2]$.
Assuming that \eqref{eqn:optdiffeqn} does not hold, choose $h$ to be of the
same sign of $m\beta p g^{p-1}+g''$ in  $t\in[t_1,t_2]$ and zero everywhere
else.
Then, for $\eps$ small enough, $K(g+\eps h,1)>K(g,1)$  and
$K(g+\eps h,t)\ge0$ for all $t$ (one simply needs to choose $\eps$ so that
$\eps g'(t)h(t)-o(\eps)<c$ for all $t\in[t_1,t_2]$.) Therefore, if $g$
does not satisfy~\eqref{eqn:optdiffeqn}, then $g+\eps h$ is a better path.
\end{proof}

Recall from (\ref{r(s)}) that
\begin{equation*}
r(s) = \left(\frac{m\beta s^2}{2}(2-p)^2\right)^{\frac{1}{2-p}}
\end{equation*}
describes the limiting shape of the boundary of the trace of the rescaled BBM. It solves
$$
\frac12r'(s)^2=m \beta r(s)^p,\quad r(0)=0.
$$
We are now going to prove that it is not possible to find a path $f$ along which the condition~\eqref{eqn:constraint} remains valid and $f(s)>r(s)$ for some $s \in(0,1).$ Indeed this would clearly yield a contradiction with Theorem \ref{rightmost_thm}.

\begin{lem}
Let $f \in C^2_{\text{piecewise}}, f(0)=0, f\ge 0$ be such that~\eqref{eqn:constraint} holds. Then
$$
f(s) \le r(s), \forall s \in [0,1].
$$
\end{lem}

\begin{proof}
Suppose that there exists $f \in C^2_{\text{piecewise}}, f(0)=0, f\ge 0$ such that~\eqref{eqn:constraint} holds and such that  $f(s_0) > r(s_0)$ for some $s_0\in (0,1).$  We construct $g \in C^2_{\text{piecewise}}, g(0)=0,g\ge 0$ as follows: let $s_1=s_0$ if $K(f,s_0)>0$ and otherwise pick $s_1>s_0$ so that $f(s_0)>r(s_1)$.
\begin{itemize}
\item[-] On $[0,s_0]$ take $g=f$,
\item[-] On $[s_0,s_1]$ take $g(s)=f(s_0)$,
\item[-] On $[s_1,1]$ $g$ is the solution of
$$
\frac12g'(s)^2=m \beta g(s)^p,\quad g(s_1)=f(s_0).
$$
\end{itemize}
Observe that $g$ has the properties that $g(1)>r(1),K(g,s_1)>0$ and $K(g,1)=K(g,s_1)$.
According to Theorem \ref{growth_paths_thm}, for any $\varepsilon>0$
$$
\liminf_{T \to  \infty} T^{-\frac{2+p}{2-p}} \log |N_T( B(g,\epsilon),1)| \ge K(g,1)=K(g,s_1)>0
$$
which contradicts Theorem \ref{rightmost_thm}.
\end{proof}

Now to prove the first line of Theorem~\ref{constrained_thm}, it is
sufficient to show that
\begin{equation}
\sup\big\{ K(f,1), 0\le f(s)\le r(s), f(1)=z \big\} \le \nas(z).
\label{1stline}
\end{equation}
Take $f$ such that $f\le r$ and $\exists t<1, K(f,t)<0$. Then $f$ cannot be
optimal among paths which stay below $r$: one can easily construct
a better path $\tilde f$ staying below $r$ by choosing $\tilde f= r$ up to
$r^{-1}( f(t) )$, constant between that point and $t$, and equal to $f$ thereafter.

\begin{lem}\label{lem:leaverightmost}
Define $s_z(g):=\inf\{s\in[0,1]:g(s)\neq r(s)\}$. Then any
optimal path $g$ solves
\begin{equation*}
m\beta p g^{p-1}+g''=0
\end{equation*}
on the interval $s\in(s_z(g),1)$. Furthermore, $K(g,s)>0$ for all $s\in(s_z(g),1)$.
\end{lem}

\begin{proof}
If $s_z(g)=1$, there is nothing to prove, so we assume $s_z(g)\in[0,1).$
The first thing to observe is that any point $(t,y) \in \{ t\in [0,1]
, 0\le y <r(t)\} $ can be reached by a path $f$ such that $f(t)=y$ and
$K(f,t)>0$. If $y>0$ just let $f = r$ until $r(s)=y$ and then let $f$ stay
constant ($K$ stays at 0 until $s$ and then accumulates some positive
growth until $t$). If $y=0$ it is equally easy to check that there is a path which reaches this point, fulfills~\eqref{eqn:constraint} and has a strictly positive $K$ at the end.

Suppose now that there exists $t>s_z$ such that $g(t)=r(t)$. Then there must exist $[a,b] \subset [s_z,t]$, $a<b$, such that $g(a)-r(a) = g(b)-r(b)=0$ and $r(t)-g(t)>0$ for all $t \in (a,b)$. But by the previous observation this means that $K(g,t)>0$ for all $t \in (a,b)$ and therefore by Lemma \ref{lem:optdiffeqn} $g$ must be strictly concave on $(a,b).$ But since $r(s)$ is strictly convex this is a contradiction. Thus we must have $K>0$ on $(s_z,1]$ and we conclude by Lemma  \ref{lem:optdiffeqn} again.
\end{proof}

Lemmas~\ref{lem:optdiffeqn},~\ref{lem:derivative_cts}
and~\ref{lem:leaverightmost} show that any positive
solution $g_z$ to the optimisation problem (in $C^2_{\text{\textup{piecewise}}}$) is such that there exists $s_z$ such that
\begin{enumerate}
\item $g_z=r$ on $[0,s_z]$,
\item $g_z$ solves $g_z''+m\beta pg_z^{p-1}=0$ on $(s_z,1]$, and $g_z(1)=z$,
\item $g_z'$ is continuous at $s_z.$
\end{enumerate}
Therefore we conclude that the unique positive solution to the optimisation
problem in $C^2_{\text{\textup{piecewise}}}$ is $g_z$.
(Note that it is easily seen that $s_z>0$ for $p>0$.)

\begin{lem}\label{lem:c2opt}
For any endpoint $z$, if $g$ is an optimal path over $C^2_{\text{\emph{piecewise}}}$ then $g$ is also optimal over $H_1$.
\end{lem}

\begin{proof}
We have shown the existence of an optimal path $g_z$ over all paths in $C^2_{\text{piecewise}}$; suppose there is a better path $f\in H_1$. That is, $f(0)=0$, $f(1)=z$, $f\geq0$, $K(f,t)\ge 0$
$\forall t\in[0,1]$ and $K(f,1) > K(g_z,1) + \varepsilon$ for some
$\varepsilon>0$. For each $n\in \N$, define $h_n : [0,1]\to\R$ by setting $h_n(k/n) = f(k/n)$ for all $k=0,1,2,\ldots,n$ and interpolating linearly elsewhere. Then each $h_n$ is piecewise linear and hence certainly piecewise $C^2$; and since $h_n$ agrees with $f$ at each $k/n$ and linear functions minimise derivatives, we have
\[\int_0^t h_n'(s)^2 \diffd s \leq \int_0^t f'(s)^2 \diffd s \hs\hs \forall t\in\left\{0,\frac1n, \frac2n,\ldots 1\right\}.\]
Now by choosing $n$ large we may insist that
\[\int_0^t h_n'(s)^2 \diffd s < \int_0^t f'(s)^2 \diffd s + \varepsilon \hs\hs \forall t\in[0,1]\]
and
\[\int_0^t h_n(s)^p \diffd s > \int_0^t f(s)^p \diffd
s - \frac\varepsilon{2m\beta} \hs\hs \forall t\in[0,1];\]
but then $h_n$ is a function in $C^2_{\text{piecewise}}$ satisfying all the
required properties and with $K(h_n,1) > K(f,1)-\varepsilon > K(g_z,1)$.
This contradicts the assumption that $g_z$ was optimal in $C^2_{\text{piecewise}}$.
\end{proof}

All that is left to prove the first part of Theorem \ref{constrained_thm} is to show uniqueness in $H_1$.
\begin{proof}[Proof of uniqueness]
Now fix $z>0$ and suppose that $\exists h \in H_1$ such  that
$h(1)=z$ and $K(h,1) = K(g_z,1)$ and
$\theta_0(h) \ge 1$ but that $h\not\equiv g_z$. Take
$s\in[0,1]$ such that $h(s)\neq g_z(s)$. By rescaling time by $1/s$ and
considering the endpoint $h(s)$, by Lemma \ref{lem:c2opt} we can find some positive piecewise $C^2$ function $f_1$ ending at $h(s)$ with growth rate $K(f,s) \geq K(h,s)$. Equally, there is an optimal positive piecewise $C^2$ function $f_2$ amongst functions beginning at $h(s)$ and ending at $z$; this is not
immediate from our results as we have not considered starting from anywhere
other than the origin, but our proofs easily carry over with no extra work.
Since these two optimal growth rates (from 0 to $h(s)$, and from $h(s)$ to
$z$) are achieved by positive piecewise $C^2$ functions, there is a positive piecewise $C^2$
function $f$ such that $f(s) = h(s)$ and $K(f,1) \geq K(h,1) = K(g_z,1)$.
This contradicts the uniqueness of $g_z$ amongst positive piecewise $C^2$ functions.

Note that for $z=0$ and $p>0$, the solution is not unique: the positive function $g_0$ and the negative function $-g_0$ are both optimal.
\end{proof}


We now turn to the second part of Theorem \ref{constrained_thm} which concerns the total population size.

\begin{lem}\label{lem:optgrowth}
There exists a unique $\bestzas\ge0$  such that
\[\bestnas:=\nas(\bestzas)=\sup_z \nas(z)=\sup \big\{ K(f,1), f\in C[0,1], \theta_0(f)=\infty\big\}.\]
The total population size satisfies
\[\lim_{T\to\infty}\frac{1}{T^{\frac{2+p}{2-p}}}\log |N(T)| = \bestnas\qquad\text{almost surely},
\]
where one finds
\[ g_{\bestzas}'(1)=0,
\qquad \bestzas = \left[\frac{\sqrt{2m\beta}}{\frac {2^{\frac{3p-2}{2p}}}{2-p}
	+\displaystyle\int_{2^{-1/p}}^1\frac{\diffd x}{\sqrt{1-x^p}}}\right]^\frac2{2-p}
\]
and
\[\bestnas=\frac{2-p}{2+p}m\beta\bestzas^p.\]
\end{lem}

\begin{proof}
By Lemma \ref{lem:c2opt} we may assume without loss of generality that $g\in C^2_{\text{piecewise}}$. Let $\eps>0$ be small and $g$ be any such function satisfying~\eqref{eqn:constraint}. For $t\in[1-\eps,1]$ we
have $g(t)=g(1)+O(\eps)$, and $g'(t)=g'(1)+O(\eps)$. Therefore
\begin{equation}\label{eqn:k}
\int_{1-\eps}^1 \left(m\beta g(t)^p-\frac12 g'(t)^2\right)\diffd
t=\eps\Big(m\beta g(1)^p-\frac12 g'(1)^2\Big)+o(\eps).
\end{equation}
Now consider the function
$$
\tilde g(t):=
\begin{cases}
g(t),\quad\quad& t\in[0,1-\eps)\\
g(1-\eps), & t\in[1-\eps,1].
\end{cases}
$$
Note that $\tilde g$ satisfies~\eqref{eqn:constraint} because $g$ does, and
\begin{align}
K(\tilde g,1)&=K(g,1-\eps)+\eps m\beta g(1-\eps)^p\nonumber\\
&= K(g,1-\eps)+\eps m\beta g(1)^p+o(\eps)\nonumber\\
&= K(g,1)+\frac\eps2 g'(1)^2+o(\eps).\label{eqn:k'}
\end{align}
But comparing~\eqref{eqn:k} and~\eqref{eqn:k'}, we see that
$K(\tilde g,1)>K(g,1)$ for all sufficiently small
$\eps$, unless $g'(1)=0$, hence the first part of the result.

\medskip

For $s>s_z$ one has
\begin{equation}
\frac12 g_z'(s)^2+m\beta g_z(s)^p=m\beta A
\end{equation}
where $A$ is some constant. For $z=\bestzas$ one has $g'_{\bestzas}\ge0$, $g_{\bestzas}'(1)=0$
and $A=\bestzas^p$. Therefore (for $s>s_z$),
\begin{equation}
\frac{g'_{\bestzas}(s)}{\sqrt{\bestzas^p-g_{\bestzas}(s)^p}}=\sqrt{2m\beta}
\label{38}
\end{equation}
We integrate the above expression with respect to $s$ from $s_{\bestzas}$ to 1 and we make the change of variable
$x=g_{\bestzas}(s)/\bestzas$. We obtain
\begin{equation}
\bestzas^{1-p/2}\int_{2^{-1/p}}^1\frac{\diffd x}{\sqrt{1-x^p}}=\sqrt{2m\beta}(1-s_{\bestzas})
\end{equation}
For the lower limit on the integral, we used that $\frac12 g_z'(s_z)^2=m\beta g_z(s_z)^p$ and hence
$g_z(s_z)^p =A/2$. For $z=\bestzas$ this gives $g_z(s_{\bestzas})=\bestzas2^{-1/p}$.
The value of $s_{\bestzas}$ then comes from the explicit expression
$g_z(s)=r(s)$ when $s\le s_{\bestzas}$.

The expression for $\bestnas$ comes from Theorem~\ref{constrained_id}, which will be proved in Section 8.
\end{proof}

The solution to the unconstrained optimisation problem given in Theorem \ref{unconstrained_thm} is now simple in light of the work above.

\begin{proof}[Proof of Theorem \ref{unconstrained_thm}]
Inspecting the proof of Lemma~\ref{lem:optdiffeqn}, we see that the
solution to this unconstrained problem amongst functions which are
piecewise $C^2$ is given by the function $h_z$ satisfying
\begin{align*}
h''(u)+m\beta ph^{p-1}&=0,\\
h(0)&=0,\\
h(1)&=z.
\end{align*}
A similar argument to that in Lemma \ref{lem:c2opt} then shows that this function is also optimal over $H_1$, and another similar to that in Lemma \ref{lem:optgrowth} gives that the optimal $z$ is that with $h_z'(1)=0$.

The value of $\bestzexp$ is obtained in the same way as $\bestzas$ except
that \eqref{38} must be integrated from 0 to 1 rather than from $s_z$ to~1.
The value of $\bestnexp$ also comes from Theorem~\ref{constrained_id}.
\end{proof}

As an aside, we note that
\begin{equation}
\nexp(0)=2^{-\frac{2p}{2-p}}\bestnexp.
\label{nexp0}
\end{equation}
This can be seen by remarking that the optimal path $h_0$ is symmetrical
around $s=1/2$ and hence that $h'_0(1/2)=0$. The trajectory $h_0$ up to
$s=1/2$ is therefore the trajectory maximising the total population at time
$T/2$ and, given the total population growth rate,
$\nexp(0)=K(h_0,1)=2\times K(h_0,1/2)=2\times
2^{-\frac{2+p}{2-p}}\bestnexp$ which is the same as \eqref{nexp0}.

\section{Further properties of the optimal paths}\label{further_optpaths_sec}

In this section we will prove Theorem~\ref{constrained_id} which states
that both growth rate $\nexp(z)$ and $\nas(z)$ are solutions of the same
differential equation~\eqref{eqK}. We will give only one demonstration for
both quantities, highlighting where necessary the differences between the two
cases. We write in a generic way $K(z)$ for either quantity $\nexp(z)$
and $\nas(z)$. Similarly, $f_z(s)$ stands in this section for either
optimal path $g_z(s)$ or $h_z(s)$ defined respectively in
Theorems~\ref{unconstrained_thm} and~\ref{constrained_thm}. One can write
\begin{equation}
K(z) =
    -\frac12\int_{s_z}^1 f'_z(s)^2\,\diffd s
+m\beta \int_{s_z}^1 f_z(s)^p\,\diffd s
.
\label{Kf}
\end{equation}
where $f_z(s)$ is a solution of
\begin{equation}
f_z''(s)+m\beta p f_z(s) ^{p-1}=0,\qquad f_z(1)=z
\label{eqf}
\end{equation}
and
\begin{equation}
\begin{cases}
s_z=0,\quad f_z(s_z)=0&\text{in the expectation case},\\
f_z(s_z)=r(s_z),\quad f'_z(s_z)=r'(s_z)&\text{in the almost sure case}.
\end{cases}
\label{inif}
\end{equation}
(In the almost sure case, the optimal path is equal to $r(s)$ for $0\le
s\le s_z$
where $s_z$ is a unknown quantity which has to be
solved for. Recall that $r(s)$ is the trajectory of the
almost sure rightmost particle; on this trajectory the population does not
grow, which is why one can start the integrals in~\eqref{Kf} from $s_z$.)

\vspace{3mm}

\noindent
We begin with a simple lemma showing monotonicity of the optimal paths in $z$.

\begin{lem}\label{z_mono}
The optimal paths and their derivatives are monotone in $z$. That is, if $0\leq w\leq z$,
\[f_w(s)\leq f_z(s) \hs\hs\text{and}\hs\hs f'_w(s)\leq f'_z(s).\]
\end{lem}

\begin{proof}
Suppose there exists $u_0$ such that $f_w(u_0)>f_z(u_0)$. By the intermediate value theorem there exists $s\in(s_z\vee s_w, u_0)$ and $t\in (u_0,1)$ such that $f_w(s)=f_z(s)$ and $f_w(t)=f_z(t)$. Then by our characterisation of the optimal functions in terms of solutions to differential equations, we must have $f_w(u) = f_z(u)$ for all $u\in[s,t]$. This contradicts the fact that $f_w(u_0)>f_z(u_0)$. A similar proof works for $f'_z$ by considering hypothetical points $u_1<u_2$ such that $f_w(u_2) - f_w(u_1) > f_z(u_2) - f_w(u_1)$.
\end{proof}

\begin{lem}\label{prime}
One has $K'(z)=-f_z'(1)$.
\end{lem}

\begin{proof}
By Lemma \ref{z_mono} we may differentiate \eqref{Kf} with respect to $z$. One gets
\begin{equation}
\begin{split}
K'(z)=
- \int_{s_z}^1 f'_z(s)\frac{\partial f'_z}{\partial z}(s)\,\diffd s
+m\beta p \int_{s_z}^1 f_z(s)^{p-1}\frac{\partial f_z}{\partial z}(s)\,\diffd s
\\
-\frac{\diffd s_z}{\diffd
z}\left[m\beta f_z(s_z)^p-\frac12f'_z(s_z)^2\right].
\end{split}
\end{equation}
The third term in the right-hand side is zero, either because $\diffd s_z/\diffd
z=0$ (expectation case) or because the square bracket is zero (almost
sure case, see \eqref{eqr(s)} with \eqref{inif}). Integrating the
first term by parts leads to
\begin{equation}
K'(z)=
\int_{s_z}^1 \left[f''_z(s)+m\beta p  f_z(s)^{p-1}
  \right]
\frac{\partial f_z}{\partial z}(s)\,\diffd s
-f'_z(1)\frac{\partial f_z}{\partial z}(1)+f'_z(s_z)\frac{\partial
f_z}{\partial z}(s_z).
\end{equation}
The integral is null because of \eqref{eqf}. As $f_z(1)=z$ for all $z$, one has
$\partial f_z/\partial z(1)=1$ in the second term. The third term is also
null because $\partial f_z/\partial z(s_z)=0$. This is trivial in the
expectation case as $f_z(s_z)=f_z(0)=0$ and it is also true in the almost sure case because\footnote{To be more verbose, as
$f_z(s_z)=r(s_z)$ one has
$$\frac{\diffd}{\diffd z}
f_z(s_z)= r'(s_z)\frac{\diffd s_z}{\diffd z}
 =  \frac{\partial f_z}{\partial z}(s_z)+
f'_z(s_z)\frac{\diffd s_z}{\diffd z}.
$$
Using $f'(s_z)=r'(s_z)$ gives $\partial f_z/\partial z(s_z)=0$.}
$f_z(s)$ is independent of $z$ up to $s=s_z$.
\end{proof}

\begin{proof}[Proof of Theorem \ref{constrained_id}]
We now establish relations between the integrals in \eqref{Kf}. Integrating the second integral by parts and applying Lemma \ref{prime}, one gets
\begin{equation}
 \int_{s_z}^1f'_z(s)^2\,\diffd s=-z K'(z)-f_z(s_z)f_z'(s_z)
-\int_{s_z}^1f_z(s)f''_z(s)\,\diffd s.
\end{equation}
Then, using \eqref{eqf},
\begin{equation}
 \int_{s_z}^1f'_z(s)^2\,\diffd s
-m\beta p \int_{s_z}^1f_z(s)^p\,\diffd s
=-z K'(z)-f_z(s_z)f_z'(s_z).
\label{lin1}
\end{equation}
We now multiply \eqref{eqf} by $f'_z(s)$ and integrate:
\begin{equation}
\frac12f_z'(s)^2+m\beta f_z(s)^p = c.
\label{withCste}
\end{equation}
The integration constant $c$ can be obtained by evaluating at $s=1$ or at $s=s_z$:
\begin{equation}
c=\frac12K'(z)^2+m\beta z^p=\frac12f'_z(s_z)^2+m\beta f_z(s_z)^p.
\end{equation}
Then, integrating \eqref{withCste} between $s_z$ and 1,
\begin{multline}
\frac12\int_{s_z}^1f'_z(s)^2\,\diffd s
+m\beta\int_{s_z}^1f_z(s)^p\,\diffd s
=(1-s_z)c
\\
=\frac12K'(z)^2+m\beta z^p-s_z\left[\frac12f'_z(s_z)^2+m\beta f_z(s_z)^p\right].
\label{lin2}
\end{multline}
where both expressions for $c$ were used.
Now, the right linear combination of \eqref{lin1} and \eqref{lin2}
gives $K(z)$:
multiplying \eqref{lin1} by $-2/(2+p)$ and \eqref{lin2} by $(2-p)/(2+p)$, adding, and using (\ref{Kf}), one gets
\begin{multline}
K(z) = \frac{2z K'(z)+(2-p)\big[K'(z)^2/2+m\beta z^p\big]}{2+p}
\\+\frac {2 f_z(s_z) f_z'(s_z)-(2-p)s_z\big[f'_z(s_z)^2/2+m\beta
f_z(s_z)^p\big]}{2+p}.
\label{endK}
\end{multline}
The second term in the right-hand side is zero. This is trivial in the expectation
case as $s_z=0$ and $f(s_z)=0$, and can easily be shown in the almost
sure case: replace all $f_z$ by $r$ using \eqref{inif}, replace $m\beta
r(s_z)^p$ by another $r'(s_z)^2/2$ from \eqref{eqr(s)} and after
simplification one gets that the term is zero if
$r'(s_z)/r(s_z)=2/[(2-p)s_z]$, which is true as can be seen from differentiating the
logarithm of \eqref{r(s)} with respect to $s$. Then, one checks easily that \eqref{endK} is
equivalent to \eqref{eqK}.

Note that for the optimal endpoint $\hat z$ one
has $K'(\hat z)=0$ and hence obtains
\begin{equation}
K(\hat z)=\frac{2-p}{2+p}m\beta\hat z^p,
\end{equation}
as stated in Theorems~\ref{unconstrained_thm} and~\ref{constrained_thm}.

It now remains to prove that \eqref{eqK} can be rewritten, for $z\ge0$, as
in \eqref{eqK2} with a plus sign in front of the square root. For this, it
is sufficient to show that $K'(0^+)\ge0$.

For $p>0$, we have seen both in the expectation and the almost-sure
case that the non-negative optimal path $f_0(s)$ going to the origin is not
identically zero but goes away from the origin to take advantage of higher
branching rates. Consider this positive optimal path and let $z_m=\max_s
f_0(s)>0$. For $0<z\le z_m$, consider the path $\tilde f$ which is equal to $f_0$ up
to the point where $z$ is reached for the last time and which is
identically equal to $z$ past that point. Clearly, one has $K(\tilde
f,1)>K(f_0,1)$ and $\theta_0(\tilde f)\ge \theta_0(f_0)$. This implies that
$K(z)> K(0)$ for all positive $z$ smaller than $z_m$ and hence that
$K'(0^+)\ge0$. (Note that this argument holds both in the expectation and
almost-sure cases, although the value of $z_m$ is not the same.) For $p=0$,
optimal paths in expectation and almost-sure cases coincide, one
has the explicit solution $K(z)=m\beta -z^2/2$ and $K'(0)=0$, the square
root in \eqref{eqK2} is zero and the sign is of no importance.
\end{proof}

\bibliographystyle{plain}
\def\cprime{$'$}

\end{document}